\documentclass[sn-mathphys,Numbered]{sn-jnl}

\usepackage{graphicx}%
\usepackage{multirow}%
\usepackage{amsmath,amssymb,amsfonts}%
\usepackage{amsthm}%
\usepackage{mathrsfs}%
\usepackage[title]{appendix}%
\usepackage{xcolor}%
\usepackage{textcomp}%
\usepackage{manyfoot}%
\usepackage{booktabs}%
\usepackage{algorithm}%
\usepackage{algorithmicx}%
\usepackage{algpseudocode}%
\usepackage{listings}%
\usepackage{soul}%
\usepackage{enumerate}
\usepackage[capitalize]{cleveref}
\usepackage[normalem]{ulem}

\makeatletter
\newtheorem*{rep@theorem}{\rep@title}
\newcommand{\newreptheorem}[2]{%
\newenvironment{rep#1}[1]{%
 \def\rep@title{#2 \ref{##1}}%
 \begin{rep@theorem}}%
 {\end{rep@theorem}}}
\makeatother
\theoremstyle{thmstyleone}%
\newtheorem{theorem}{Theorem}[section]

\newreptheorem{theorem}{Theorem}
\newtheorem{proposition}[theorem]{Proposition}%
\newtheorem{lemma}[theorem]{Lemma}%
\newtheorem{corollary}[theorem]{Corollary}%
\newtheorem{conjecture}[theorem]{Conjecture}%
\newtheorem{remark}[theorem]{Remark}%

\newcommand{\CR}{\mathcal{C}(\mathbb{R}^2)}

\theoremstyle{thmstyletwo}%
\newtheorem{example}{Example}%
\newtheorem{assumption}{Assumption}%

\theoremstyle{thmstylethree}%
\newtheorem{definition}{Definition}%

\begin{document}

\title[Article Title]{Quantitative Transversal Theorems in the Plane}


\author[1]{\fnm{Ilani} \sur{Axelrod-Freed}}\email{ilani\_af@mit.edu}

\author[2]{\fnm{Jo\~ao} \sur{Pedro Carvalho}}\email{jpcarv@umich.edu}

\author*[3]{\fnm{Yuki} \sur{Takahashi}}\email{yukit@berkeley.edu}

\affil[1]{\orgdiv{Mathematics Department}, \orgname{Massachusetts Institute of Technology}, \orgaddress{\city{Cambridge}, \postcode{02461}, \state{MA}, \country{USA}}}

\affil[2]{\orgdiv{Mathematics Department}, \orgname{University of Michigan}, \orgaddress{\city{Ann Arbor}, \postcode{48104}, \state{MI}, \country{USA}}}

\affil*[3]{\orgdiv{Mathematics Department}, \orgname{University of California, Berkeley}, \orgaddress{\city{Berkeley}, \postcode{94720}, \state{CA}, \country{USA}}}


\abstract{Hadwiger's theorem is a Helly-type theorem involving common transversals to families of convex sets instead of common intersections. Subsequently, Pollack and Wenger identified a necessary and sufficient condition, called a consistent $k$-ordering, for the existence of a hyperplane transversal for sets in $\mathbb{R}^d$.
We obtain a quantitative generalization of Hadwiger's theorem in $\mathbb{R}^2$, showing that compact convex sets in $\mathbb{R}^2$ with a quantitative version of consistent ordering have a transversal satisfying quantitative requirements. Our proof generalizes the methods in Wenger's proof of Hadwiger's theorem in $\mathbb{R}^2$. We also prove colorful versions of our results.}




\maketitle

\section{Introduction}\label{sec1}

Helly's theorem \cite{Helly:1923wr} is one of the most fundamental results in discrete geometry. It states that \textit{given a finite family of convex sets in $\mathbb{R}^d$, if every $d+1$ sets have a point in common, then the whole family has a point in common.} 
In this paper, we investigate a transversal variant of Helly's theorem, called Hadwiger's theorem.
Hadwiger's theorem \cite{Hadwiger:1957tx} states that \textit{given an ordered family of pairwise disjoint compact convex sets in the plane, if every three sets can be intersected by some directed line (a transversal) consistent with the ordering, then there exists a common transversal of the family.} In particular, we prove quantitative and colorful versions of Hadwiger's theorem, motivated by quantitative and colorful generalizations of Helly's theorem.

Building on the results of Katchalski \cite{Katchalski1980} and 
Goodman and Pollack \cite{Goodman1988}, 
Wenger \cite{Wenger.1990} provided a proof of Hadwiger's theorem which holds for non-disjoint compact convex sets. Since then Pollack and Wenger \cite{Pollack:1990cna} were able to prove similar results for (non-disjoint) compact connected sets in $\mathbb{R}^d$ by abstracting the requirement of any $d+1$ sets having certain transversals to the notion of having a consistent ordering (also called ``separating consistently" in \cite{McGinnis2024}).
\begin{definition}[\cite{Amenta:2017ed}]\label{def:cons-order}
    Let $\mathcal{F}$ be a finite family of connected sets in $\mathbb{R}^d$. We say that there is a \textcolor{blue}{consistent $k$-ordering of $\mathcal{F}$} 
    if there exists a map $\phi:\mathcal{F}\to \mathbb{R}^k$ such that for any two subfamilies $\mathcal{F}_1, \mathcal{F}_2$ with $|\mathcal{F}_1|+|\mathcal{F}_2|\leq k+2$,
    \[\text{conv}(\mathcal{F}_1)\cap \text{conv}(\mathcal{F}_2)=\emptyset \implies \text{conv}(\phi(\mathcal{F}_1))\cap \text{conv}(\phi(\mathcal{F}_2))=\emptyset.\]
\end{definition}
\noindent For a family $\mathcal{F}$ of (non-disjoint) compact convex sets in $\mathbb{R}^2$, the existence of a consistent 1-ordering of $\mathcal{F}$ is equivalent to having an order on $\mathcal{F}$ such that any three sets have a transversal that stabs them in order (\cite{Wenger.1990}, \cite{Pollack:1990cna}, see \cref{def:stabinorder} in \cref{sec:StabbingOrder}).

Another generalization of Helly-type results is to the quantitative case, which imposes requirements on how ``large" the intersection of sets must be.
Previous quantitative Helly-type results characterized this notion of largeness by the width (\cite{Buchman1982}), diameter (\cite{Barany:1984ed}, \cite{Silouanos2017}, \cite{Dillon2021}, \cite{Hernandez2022}), and volume (\cite{Naszodi:2016he}, \cite{Sarkar:2019tp}, \cite{Jung2022}) of convex sets. 
For a survey of this type of results, see \cite{Holmsen2017}.
As a way to combine various notions of largeness, Amenta, De Loera, and Sober\'{o}n \cite{Amenta:2017ed} suggest measuring the largeness of an intersection using non-negative monotone functions on the space of all convex sets in $\mathbb{R}^d.$
We expand upon this viewpoint by providing a quantitative notion for transversals. That is, given a monotone function $f$ and a constant $\alpha$,
we consider a transversal $\ell$ to cut a convex set $C$ in ``large enough" pieces if $f(C\cap\ell^+), f(C\cap\ell^-) \geq \alpha$ where $\ell^+$ and $\ell^-$ are the left and right half-planes of $\ell$. 
We refer to a transversal with this property as an $(f,\alpha)$-transversal (see \cref{def:f-alpha-transversal}). 

In this paper, we provide a quantitative version of Hadwiger's theorem guaranteeing the existence of an $(f, \alpha)$-transversal. We adapt proof strategies from Wenger (\cite{Wenger.1990}), but with a new framing in terms of the existence of a consistent 1-ordering. To do this, we define a quantitative version of the consistent 1-ordering, which we call an $(f,\alpha)$-consistent ordering. However, unlike the non-quantitative case, there can now exist three sets with an $(f,\alpha)$-transversal but no  $(f,\alpha)$-consistent ordering.

Another direction for the generalizations of Helly-type theorems is colorful variations.
For instance, both Helly's theorem and Carath\'eodory's theorem have colorful generalizations \cite{BARANY1982141}.
As for Hadwiger's theorem, Arocha, Bracho, and Montejano \cite{Arocha:2008ti} proved a colorful version of Hadwiger's theorem in $\mathbb{R}^2$. 
More recently, Holmsen \cite{Holmsen2022}, based on \cite{Holmsen2016}, showed a colorful version of Hadwiger's theorem in $\mathbb{R}^d$. 
In this paper, we follow Arocha, Bracho, and Montejano \cite{Arocha:2008ti}'s approach to prove colorful and quantitative Hadwiger in $\mathbb{R}^2$. 

We now introduce the structure of this paper.
In \cref{stab separate direction}, we define $(f,\alpha)$-stabbing and -separation directions and establish properties of $(f, \alpha)$-transversals. We also provide the definition of an $(f,\alpha)$-consistent ordering.

In \cref{Wenger thrms}, we prove the following quantitative version of Hadwiger's theorem.
\begin{theorem}
\label{thrm:quantitativehadwiger}
Let $\mathcal{F}$ be a finite family of compact, convex sets in the plane with an $(f,\alpha)$-consistent ordering. Then, there exists an $(f,\alpha)$-transversal for all sets in $\mathcal{F}$. 
\end{theorem}

In \cref{sec:StabbingOrder}, we provide an example of three sets with an $(f,\alpha)$-transversal that cannot have an $(f,\alpha)$-consistent ordering. We show that if we impose certain additional conditions on the sets, then we can define an $(f,\alpha)$-stabbing order for each transversal. 
In such a setting, the following are equivalent:
\begin{enumerate}
    \item existence of an $(f,\alpha)$-consistent ordering of $\mathcal{F}$,
    \item existence of some ordering of $\mathcal{F}$ so that every three sets in $\mathcal{F}$ has an $(f,\alpha)$-transversal consistent with the ordering.
\end{enumerate}
In this setting, we prove quantitative versions of two additional theorems of Wenger (\cite[Theorem 2, Theorem 3]{Wenger.1990}) guaranteeing when an ordered family of sets has a transversal that stabs them all \emph{in order}.

Finally, in \cref{colorful} we prove \cref{thm:colorfulquanthadwiger}, providing a colorful version of \cref{thrm:quantitativehadwiger} based on the work of Arocha, Bracho, and Montejano \cite{Arocha:2008ti}. For this, we define a colorful $(f,\alpha)$-consistent ordering (see \cref{defn:colorful_consistent_ordering}). Our definition generalizes the notion of ``rainbow consistent $k$-ordering" in \cite{Holmsen2016}.

\begin{theorem}\label{thm:colorfulquanthadwiger}
Let $\mathcal{F}$ be a finite family of compact, convex sets in the plane, where each set is colored red, green, or blue, with a colorful $(f,\alpha)$-consistent ordering.
Then, there exists a color with an $(f,\alpha)$-transversal to all the sets of that color.
\end{theorem}

\section{$(f, \alpha)$-stabbing and -separation directions}
\label{stab separate direction}
In this section, we begin by defining the monotone functions we will be working with, from which we define $(f,\alpha)$-traversals, stabbing and separation directions. Using these, we define the notion of an $(f,\alpha)$-consistent ordering, and establish some foundational lemmas.

To make things precise, we first define the quantitative metric on sets which we will be working with. 
Let \textcolor{blue}{$\CR$} be the space of compact, convex subsets of $\mathbb{R}^2$.

\begin{definition}
A \textcolor{blue}{monotone function} on convex sets is a function $f: \mathcal{C}(\mathbb{R}^2) \rightarrow \mathbb{R}$, such that for any convex sets $C_1 \subseteq C_2 \in \mathcal{C}(\mathbb{R}^2)$, we have $f(C_1) \leq f(C_2)$.
\end{definition}

\begin{assumption}\label{remark:hasudorffmetric}
For the rest of this paper, we will also assume that all our monotone functions are continuous with respect to the Hausdorff metric on the space of compact sets and that $f(\emptyset)=0$.
\end{assumption}

Classic examples of monotone functions on convex sets include area and perimeter functions. Given a family of sets $C_1, \dots, C_n$ in $\CR$, we can pick a monotone function $f_{i}$ and a value $\alpha_{i}$ for each $C_i$. These functions and values may be the same for all $C_i$, or they may differ. For example, if we want to measure the percent area of each set, then we let $f_{i}$, the function measuring the percent area $C_i$, be defined on any $C\in \mathcal{C}(\mathbb{R}^2)$ by 
\[f_i(C_i) = 1, f_i(\emptyset)=0, \text{ and } f_i(C) = \frac{\text{area of }(C \cap C_i)}{\text{area of }C_i}.\]
\begin{remark}\label{remark:convex-connected}
    While we work in the space of compact convex sets of $\mathbb{R}^2$, all of the results in this paper, except for special cases mentioned in \cref{rmk:not convex weird}, still hold for compact (not necessarily convex or even connected) sets in $\mathbb{R}^2$. In that case, we require that each $f_i$ is monotone and continuous (with respect to the Hausdorff metric) on the relevant subsets of $C_i$, which are exactly all intersections of $C_i$ with closed half-planes. 
\end{remark} 

Now, instead of asking whether a line $\ell$ intersects a set, we will be concerned with the value that our monotone function takes on the subsets on either side of $\ell$. We will say that $\ell$ ``stabs" a set $C$ if the value of our corresponding monotone function $f_{i}$ on the subsets on both sides of $\ell$ is sufficiently large. We define this formally as follows:

Given a directed line $\ell \in \mathbb{R}^2$ and a set $C \in \CR$, let $\ell^+(C)$ be the intersection of $C$ and the closed half space bounded by $\ell$ containing all points to the left of $\ell$'s pointing direction, and $\ell^-(C)$ be the intersection of $C$ and the opposite closed half space bounded by $\ell$. 

\begin{definition}\label{def:f-alpha-transversal}
Let $C_1, \dots, C_n$ be sets in $\mathcal{C}(\mathbb{R}^2)$. Let $f = (f_{1}, \dots , f_{n})$ be a tuple of monotone functions $f_{i}: \mathcal{C}(\mathbb{R}^2) \rightarrow \mathbb{R}$ chosen for each $C_i$. Let $\alpha = (\alpha_{1}, \dots, \alpha_{n}) \in \mathbb{R}^n$. We say a directed line $\ell$ is an \textcolor{blue}{$(f,\alpha)$-stabber} of the set $C_i$ if $f_{i}(\ell^+(C_i)), f_{i}(\ell^-(C_i)) \geq \alpha_{i} $. 
We say $\ell$ is an \textcolor{blue}{($f,\alpha)$-transversal} of sets $C_1, \ldots, C_n$ if it is an $(f, \alpha)$-stabber of all of the sets.
\end{definition}

By slight abuse of notation, in the rest of this paper we often refer to a single function $f: \mathcal{C}(\mathbb{R}^2) \rightarrow \mathbb{R}$ and number $\alpha \in \mathbb{R}$ for all our sets instead of writing out each $f_{i}$ and $\alpha_{i}$ explicitly. In such cases, it is to be assumed that $f$ acts as $f_{i}$, and likewise that $\alpha$ takes value $\alpha_{i}$ on the set $C_i$. However for efficiency, we omit the subscripts in the rest of this paper.

\begin{definition}
Given a monotone function $f: \mathcal{C}(\mathbb{R}^2) \rightarrow \mathbb{R}$ and a number $\alpha \in \mathbb{R}$,
a direction $v \in S^1$ is an \textcolor{blue}{$(f,\alpha)$-stabbing direction} for sets $C_1,...,C_n \in \CR$ if there exists a line $\ell$ in the direction of $v$ such that 
$\ell$ is an $(f, \alpha)$-transversal of $C_1, \ldots, C_n$.
Conversely, $v$ is an \textcolor{blue}{$(f, \alpha)$-separation direction} if there exists some line $\ell$ in the direction of $v$ such that for some $1\leq i,j\leq n$ either 
$f(\ell^+(C_i)), f(\ell^-(C_j)) < \alpha$ or $f(\ell^-(C_i)), f(\ell^+(C_j)) < \alpha$. 
This line is called an \textcolor{blue}{$(f,\alpha)$-separation line}. (See \cref{fig:transversal-separation}.)
\end{definition}

\begin{assumption}
\label{rmk:all directions stab}
    For the rest of this paper, assume $\mathcal{F}$ refers to a finite family of sets in $\CR$ equipped with monotone function $f$ and value $\alpha$ which satisfy the following:
   For all $C \in \mathcal{F}$, every direction $v\in S^1$ is an $(f,\alpha)$-stabbing direction for $C$.
\end{assumption} 

 Then for any line $\ell$ and set $C \in \mathcal{F}$, at most one of $\ell^+(C)$ and $\ell^-(C)$ is less than $\alpha$. 
If $\ell^+(C) < \alpha$ we say $C$ is \textcolor{blue}{on the right} side of $\ell$, and if $\ell^-(C) < \alpha$ we say it is \textcolor{blue}{on the left} side.
We say sets $A,B$ are \textcolor{blue}{$(f,\alpha)$-{separated}} if there exists an $(f,\alpha)$-separation line with $A$ on one side and $B$ on the other. 

\begin{figure}[h]
       \begin{center}
       \includegraphics[width=1\textwidth]
        {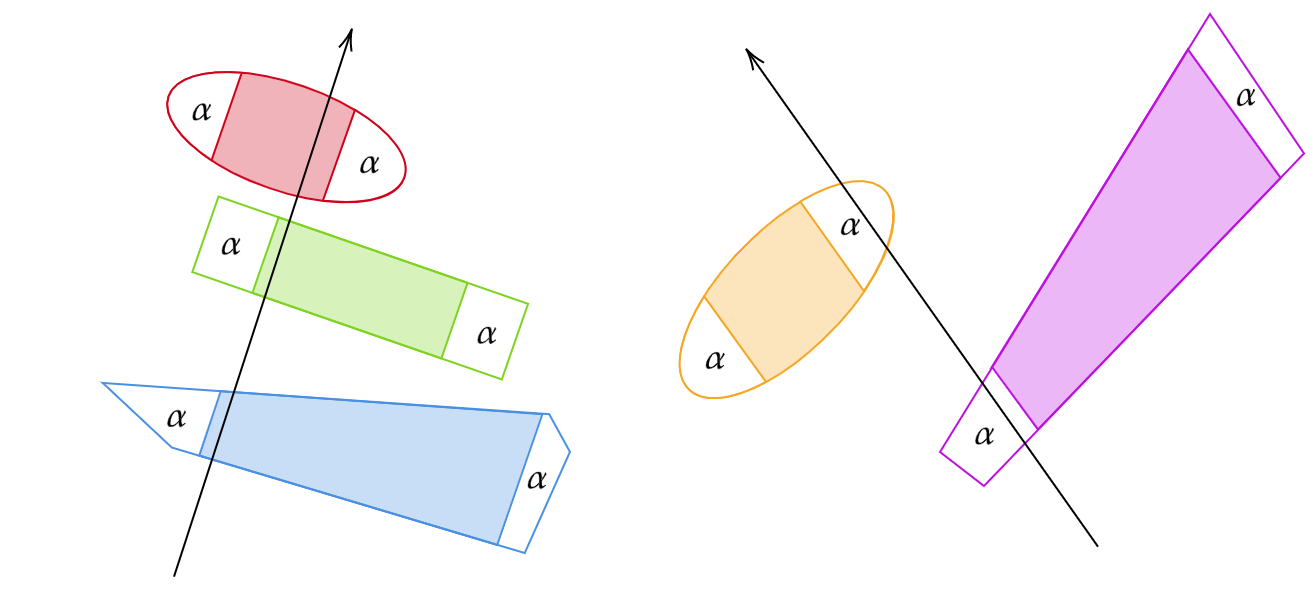}
        \caption{\label{fig:transversal-separation} Left: an $(f,\alpha)$-transversal for three convex sets. Right: an $(f,\alpha)$-separation line for two convex sets with the ellipse on the left side and the quadrilateral on the right side. (Remember that $\alpha$ represents a vector assigning an $\alpha_i$ value to each set, with these values potentially differing by set.)} 
        \end{center}
        \end{figure}

Using this notation of $(f, \alpha)$-stabbing/separating, we can now obtain the following quantitative version of \cite[Lemma 1]{Wenger.1990}. 

\begin{lemma}
\label{1dhellyuse}
 Given sets $C_1,...,C_n$ in $\mathcal{F}$, a direction $v\in S^1$ is an $(f,\alpha)$-stabbing direction for each pair of these sets if and only if it is an $(f,\alpha)$-stabbing direction for all the sets. 
\end{lemma}

\begin{proof}
The if direction is immediate.
Now assume that there exists a shared $(f,\alpha)$-stabbing direction $v$ for each pair of the sets $C_1,...,C_n$ in $\mathbb{R}^2$.
Remove from each side of the set $C_i$ the maximal convex subsets $C_i^+$ and $C_i^-$ bounded on the left and right respectively by a line in the direction of $v$ and such that $f(C_i^+),f(C_i^-)<\alpha$ (see \cref{fig:2.1proof}).
We call this new set $C_i'$. Note that for every two of the sets $C_1',...,C_n'$, the vector $v$ is still a regular stabbing direction.

Now if we project the sets $C_1',...,C_n'$ onto a line perpendicular to $v$, we get a set of line segments any two of which share a point. Helly's theorem then tells us that all the segments share a point. Extend a line $\ell$ through this point in the direction of $v$, so it intersects all the sets $C_1',...,C_n'$.
If $f(\ell^+(C_i)) < \alpha$ for some $i$, then $\ell^+(C_i)$ would have been removed from $C_i$ to obtain $C_i'$, and so $\ell$ would not intersect $C_i'$, which is a contradiction. The same argument works if $f(\ell^-(C_i)) < \alpha$ for some $i$.
Thus, we conclude that $f(\ell^+(C_i)),f(\ell^-(C_i)) \geq \alpha$ for all $i$.
\begin{figure}[h]
       \begin{center}
       \includegraphics[width=1\textwidth]
        {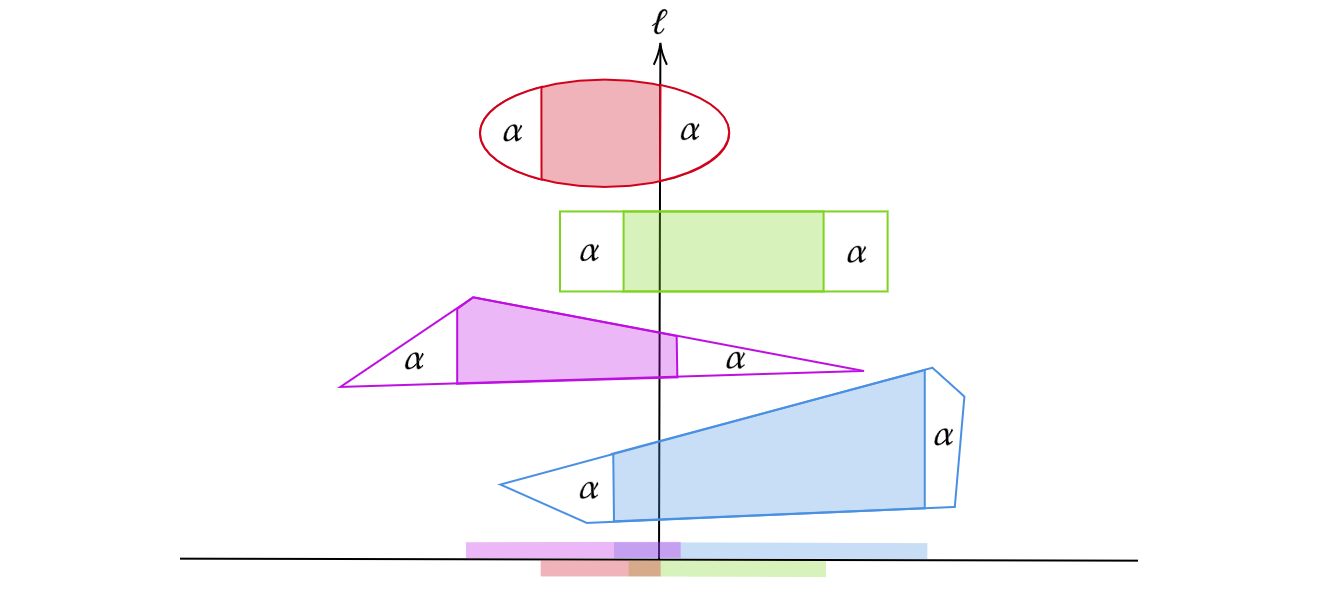}
        \caption{\label{fig:2.1proof} Convex sets $C_1,...,C_4$ with subsets $C_1',...,C_4'$ shaded in. Projections onto the perpendicular line are drawn, and the $(f,\alpha)$-transversal is going through the intersection point.}
        \end{center}
        \end{figure}
\end{proof}

\begin{corollary}
    Given sets $C_1, \dots, C_n$ in $\mathcal{F}$, each direction $v \in S^1$ is exactly one of $(f, \alpha)$-stabbing or -separation direction for all the sets.
\end{corollary}

\begin{proof}
Given a direction $v\in S^1$ and a pair of sets $C_i, C_j$ , define $C_i', C_j'$ as in the proof of \cref{1dhellyuse}. If the projections of $C_i', C_j'$ onto a line perpendicular to $v$ intersect, then $v$ is an $(f,\alpha)$-stabbing direction. Otherwise, it is an $(f,\alpha)$-separation direction. Thus, $v$ must be exactly one of $(f,\alpha)$-stabbing or $(f,\alpha)$-separation direction for each pair of sets.
If $v$ is an $(f,\alpha)$-separation direction for some pair of sets, it is an $(f,\alpha)$-separation direction for all the sets. Otherwise it is an $(f,\alpha)$-stabbing direction by \cref{1dhellyuse}.
\end{proof}

Next, we prove that the set of $(f,\alpha)$-separation directions is open. First, we define the following:

\begin{definition}
    Let $A, B \in \mathcal{F}$. Define $\textcolor{blue}{s_{AB}} \subseteq S^1$ to be the set of $(f,\alpha)$-separation directions of $A$ and $B$ for which $A$ falls on the left and $B$ falls on the right side of the corresponding $(f,\alpha)$-separation lines.
\end{definition}

\begin{remark}\label{remark:sAB-sBA}
    \cref{rmk:all directions stab} guarantees that no direction can be in both $s_{AB}$ and $s_{BA}$ for any pair of sets $A,B$.
\end{remark}

\begin{lemma}\label{lem:open}
    For any two sets $A, B\in \mathcal{F}$, the set $s_{AB}$ is open in $S^1$.
\end{lemma}

\begin{proof}
    If $s_{AB}$ is empty, it is open. Otherwise, consider a direction $v\in s_{AB}.$ We will find an open interval around $v$ in $S^1$ still contained in $s_{AB}$. We know there is some line $\ell$ in the direction of $v$ such that $A$ lies to its left and $B$ to its right. Let $A' = \ell^-(A)$ and $B'=\ell^+(B)$. Then $f(B')<\alpha$, so $f(B')=\alpha-\epsilon$ for some $\epsilon>0.$ Consider a point $p\in \ell$ such that both $A$ and $B$ lie after $p$ (when traversing $\ell$ in the direction $v$). Let $\ell_{\theta}$ be the line obtained by rotating $\ell$ clockwise by an angle of $\theta$ around the point $p$ (for some small $\theta>0$).
    See \cref{fig:2.4proof}.
    Let $A'_{\theta} = \ell_{\theta}^- (A)$ and $B'_{\theta} = \ell_{\theta}^+ (B)$. We have $A'_{\theta}\subseteq A'$ for any $\theta>0,$ so $A$ will always lie to the left of $\ell_{\theta}.$ Now, for $B'_{\theta},$ note that as $\theta \to 0,$ we have $B'_{\theta}\to B',$ and $B'\subseteq B'_{\theta}$. Then since $f$ is monotone, we have $f(B'_{\theta})\geq f(B')$. By \cref{remark:hasudorffmetric} (continuity of $f$ with respect to Hausdorff distance), we then have $f(B'_{\theta})\to f(B').$ In particular, there is some angle $\phi>0$ such that $f(B'_{\phi})<f(B')+\epsilon=\alpha$. Thus, for all $0 \leq \theta < \phi$, we have that $B$ lies to the right of $\ell_{\theta}.$

    We have an analogous argument for $\theta$ going counterclockwise around $p$, where now enough of $B'_{\theta}$ is always to the right, but we have to ensure small enough $A'_{\theta}$. By picking the smallest of the two values of $\phi$ from either direction, we get an open interval around $v$ in $S^1$ still contained in $s_{AB}.$ Thus $s_{AB}$ is open.
    \begin{figure}[h]
       \begin{center}
       \includegraphics[width=1\textwidth]
        {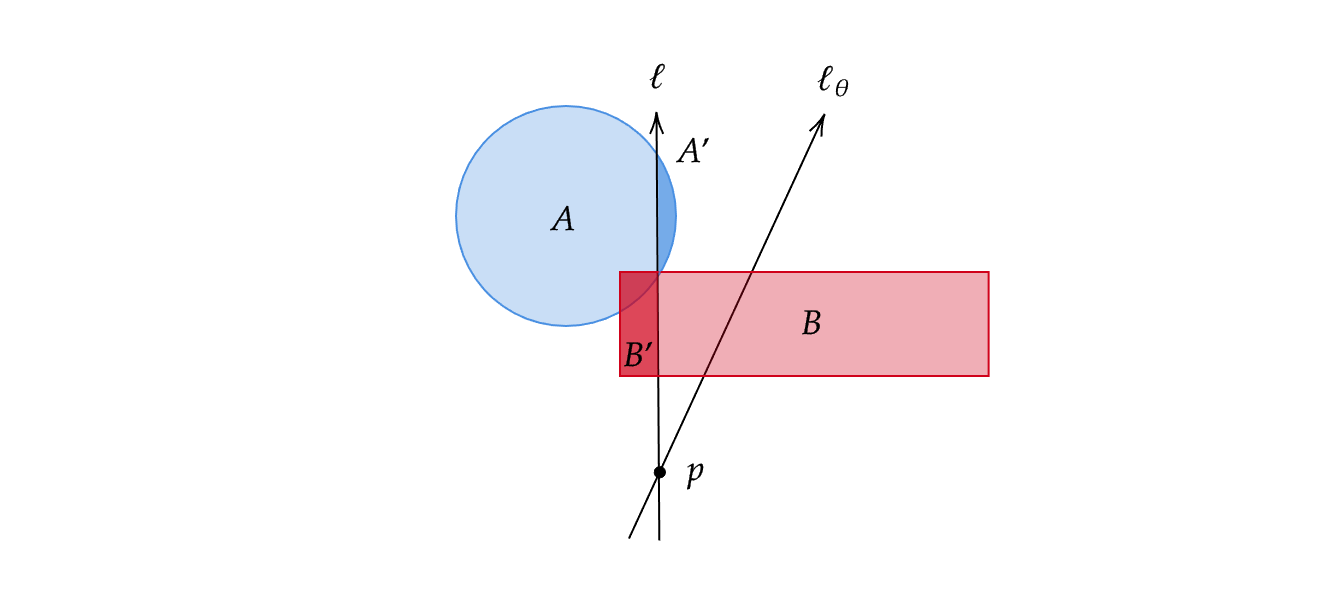}
        \caption{\label{fig:2.4proof} Sets $A$ and $B$, along with a point $p\in \ell$ and a line $\ell_\theta$, obtained by rotating $\ell$ clockwise. The darker shaded area of $A$ is $A'=\ell^-(A)$ and the darker shaded area of $B$ is $B'=\ell^+(B)$.}
        \end{center}
        \end{figure} 
\end{proof}

\begin{corollary}\label{cor:nonemptystabbingAB}
  Any two sets in $\mathcal{F}$ have an $(f,\alpha)$-transversal.
\end{corollary}

\begin{proof}
   Since $s_{AB}$ and $s_{BA}$ are both open (\cref{lem:open}) and disjoint (\cref{remark:sAB-sBA}), there exists $v\in S^1$ with $v\notin s_{AB}\cup s_{BA}$. Thus, there is a $(f,\alpha)$-transversal for $A,B$ in direction $v$.
\end{proof}

Finally, using the concept of $(f,\alpha)$-separation lines, we provide a quantitative generalization of a consistent 1-ordering from \cite{Amenta:2017ed}. 
Note that in the non-quantitative definition of a consistent 1-ordering in the plane, 
two subfamilies $\mathcal{F}_1$ and $\mathcal{F}_2$ of $\mathcal{F}$ satisfy $conv(\mathcal{F}_1) \cap conv(\mathcal{F}_2) = \emptyset$ if and only if there exists a separation line with all sets in $\mathcal{F}_1$ on one side and all sets in $\mathcal{F}_2$ on the other.

\begin{definition}\label{def:consistent-ordering}
The family $\mathcal{F}$ has an \textcolor{blue}{$(f,\alpha)$-consistent ordering} if there exists a map $\phi: \mathcal{F}\to \mathbb{R}$ such that for any subfamilies $\mathcal{F}_1, \mathcal{F}_2 \subseteq \mathcal{F}$ with $|\mathcal{F}_1|+|\mathcal{F}_2|\leq 3$,
  if there is an $(f,\alpha)$-separation line with sets in $\mathcal{F}_1$ on one side and $\mathcal{F}_2$ on the other, then
  $conv(\phi(\mathcal{F}_1)) \cap conv(\phi(\mathcal{F}_2)) = \emptyset.$
\end{definition}

The following lemma is an immediate consequence of this definition which we include here for ease of reference.
\begin{lemma}\label{lem:injectivephi}
    Suppose the family $\mathcal{F}$ has an $(f,\alpha)$-consistent ordering with $\phi:\mathcal{F}\to \mathbb{R}$. Let $A,B, C \in \mathcal{F}$.
    
    \begin{enumerate}[a.]
        \item If $s_{AB}\neq \emptyset$, then $\phi(A)\neq \phi(B).$

        \item If $\phi(A) \leq \phi(B) \leq \phi(C)$, then there is no $(f,\alpha)$-separation line with $B$ on one side and $A,C$ on the other.
    \end{enumerate}

\end{lemma}

\section{Quantitative Hadwiger}\label{Wenger thrms}

In this section, we adapt proof methods of Wenger \cite{Wenger.1990} to prove a quantitative version of \cite[Theorem 1]{Wenger.1990}.
Importantly, our \cref{thrm:quantitativehadwiger} has the hypothesis ``$(f,\alpha)$-consistent ordering" for the family $\mathcal{F}$ of compact convex sets, as opposed to ``every three sets have some  $(f,\alpha)$-transversal consistent with the prescribed ordering of $\mathcal{F}$" as in Wenger's non-quantitative version. We will explore the connection between these two hypotheses in \cref{sec:StabbingOrder}.

\begin{proof}[Proof of \cref{thrm:quantitativehadwiger}] 

Suppose our family $\mathcal{F}$ has an $(f,\alpha)$-consistent ordering given by a function $\phi: \mathcal{F} \rightarrow \mathbb{R}$. For $A, B \in \mathcal{F}$ we will denote $A<B$ if $\phi(A) < \phi(B)$.
Let $S$ be the set of all $(f,\alpha)$-separation directions for some pair of sets in $\mathcal{F}$, i.e., 
\[S := \bigcup\limits_{A, B\in \mathcal{F}} s_{AB}.\]
 Provided that $S$ does not cover the entire circle, we have some direction $v_{\text{stab}}$ which is an $(f,\alpha)$-stabbing direction for all of our sets simultaneously. Then by \cref{1dhellyuse}, there is an $(f, \alpha)$-transversal for all sets in the direction of $v_{\text{stab}}$.

To prove that $S$ does not cover the entire circle, we define two subsets
\[S_1 := \bigcup\limits_{A < B | A,B \in \mathcal{F}} s_{AB} \ \ \ \ and \ \ \ \ S_2 := \bigcup\limits_{A<B | A, B\in \mathcal{F}} s_{BA}.\] 
By \cref{lem:injectivephi} (a), the set $s_{AB}$ is empty when $\phi(A) = \phi(B)$, so $S=S_1\cup S_2$. Let $A,B \in \mathcal{F}$ such that $s_{AB} \subseteq S_1$, and let $X,Y \in \mathcal{F}$ such that $s_{YX} \subseteq S_2$. We will show $s_{AB} \cap s_{YX} = \emptyset$. For contradiction, assume this is not the case. By \cref{remark:sAB-sBA}, we cannot have both $X=A$ and $Y=B$. Thus, there are two possible cases.

\textbf{Case 1:} There exist sets $A,B,C \in \mathcal{F}$ with $\phi(B)$ falling between $\phi(A)$ and $\phi(C)$, and for which $s_{AB} \cap s_{CB} \neq \emptyset$ (note that if $s_{AB} \cap s_{CA}\neq \emptyset$, then every vector in the intersection is also in $s_{CB}$). For each vector $v\in s_{AB} \cap s_{CB}$, there exists an $(f,\alpha)$-separation line in the direction of $v$ with $B$ on one side and $A,C$ on the other. By \cref{lem:injectivephi} (b), this contradicts $\phi$ being an $(f,\alpha)$-consistent ordering of $A,B$ and $C$. 

\textbf{Case 2:}
We have distinct $A,B,X,Y \in \mathcal{F}$ with $A<B$ and $X<Y$, and such that $s_{AB} \cap s_{YX} \neq \emptyset$. For $v \in s_{AB} \cap s_{YX}$, we can find parallel lines $\ell,\ell'$ in the direction of $v$ such that 
 $A$ is on the left of $\ell$ while $B$ is on the right, and $Y$ is on the left of $\ell'$ while $X$ is on the right. (See \cref{ABYX} for positioning reference.) Here we assume $\ell$ falls to the left of $\ell'$, but the proof goes the same way if they are flipped. If $X < A$, then $X$ and $B$ are on the right of $\ell$ while $A$ is on the left. Then by \cref{lem:injectivephi} (b), $X, A, B$ cannot have an $(f,\alpha)$-consistent ordering witnessed by $\phi$ in the desired order. Similarly, if $A<X$, then $A$ and $Y$ are on the left of $\ell'$ while $X$ is on the right, so again we cannot find $\phi$ for $A,X, Y$ that is compatible with the desired order. By \cref{lem:injectivephi}(a), we can never have $\phi(A) = \phi(X)$, since $A$ and $X$ are $(f,\alpha)$-separated by $\ell$.

\begin{figure}[h]
       \begin{center}
       \includegraphics[width=.7\textwidth]
        {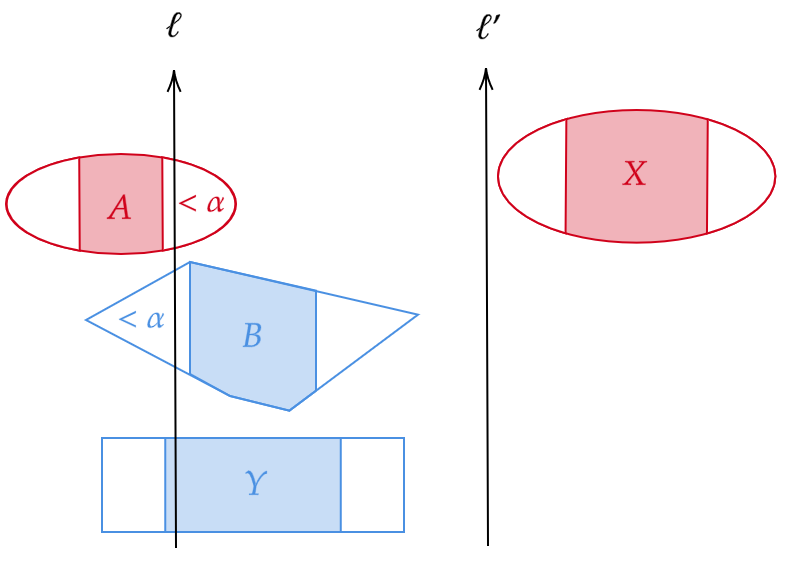}
        \caption{\label{ABYX} Sets with nonempty intersection of $S_1$ and $S_2$. The line $\ell$ separates $A$ and $B$, and the line $\ell'$ separates $X$ and $Y$. Note that no $(f,\alpha)$-transversal exists in the order provided for the sets $X$, $A$, $B$ or $A$, $X$, $Y$.} 
        \end{center}
        \end{figure}
 
Thus, we conclude that $S_1 \cap S_2 = \emptyset$. Neither $S_1$ nor $S_2$ can individually cover the entire circle, because for $v \in s_{AB} \subseteq S_1$ we have $-v \in s_{BA} \subseteq S_2$ and therefore $-v \notin S_1$. Both $S_1$ and $S_2$ are open in $S^1$ by \cref{lem:open}, and two disjoint open sets cannot cover the entire circle, so there must be some direction in $S^1$ contained in neither $S_1$ nor $S_2$ as desired (see \cref{fig:stabbingdirection}).
\begin{figure}[h]
       \begin{center}
       \includegraphics[width=1\textwidth]
        {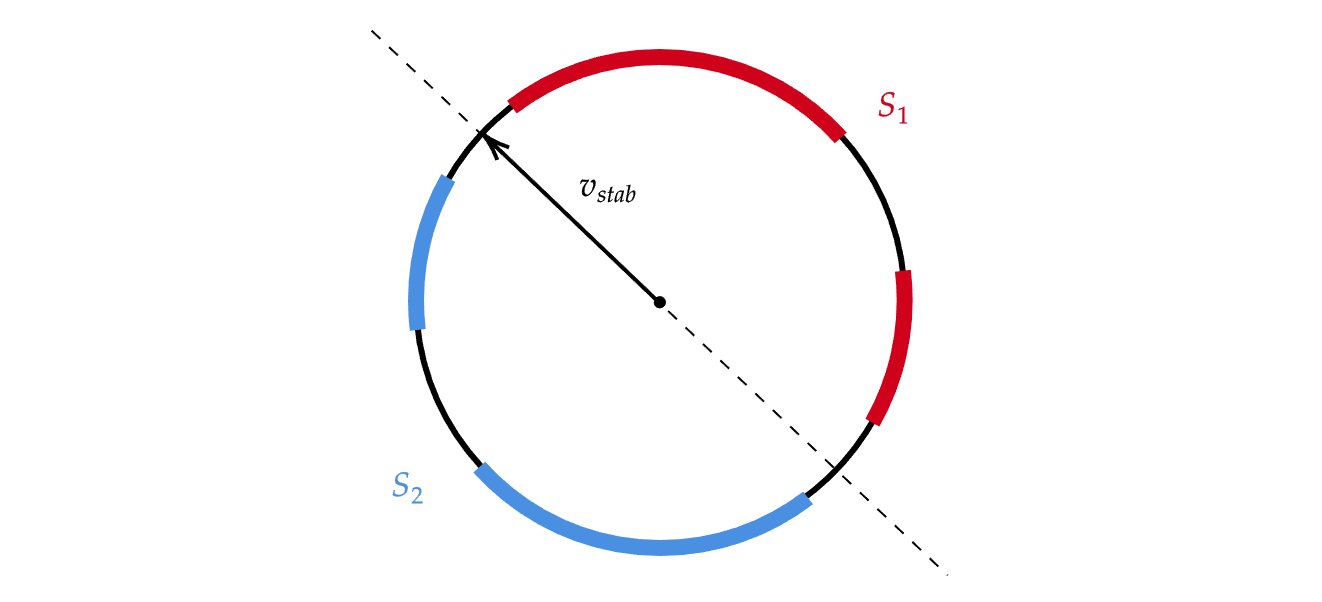}
        \caption{\label{fig:stabbingdirection} The sets $S_1$ (thick red arcs on the right side of $v_{\text{stab}}$) and $S_2$ (thick blue arcs on the left side of $v_{\text{stab}}$) for some family of convex sets drawn on a circle. The arcs of $S_1$ and $S_2$ do not intersect, so we are guaranteed stabbing directions for all sets. We picked one of these to be $v_{\text{stab}}$.} 
        \end{center}
        \end{figure}
\end{proof}

\section{$(f,\alpha)$-stabbing order}\label{sec:StabbingOrder}

\subsection{$(f,\alpha)$-consistent ordering versus $(f,\alpha)$-transversals}\label{subsec:4.1}

Wenger \cite{Wenger.1990} states the hypothesis of Theorems 1, 2 and 3 (the non-quantitative versions of our Theorems \ref{thrm:quantitativehadwiger}, \ref{thrm:6stabber}, and \ref{thrm:4stabber}) as requiring an appropriate transversal for every three sets, rather than in terms of a consistent ordering. Formally, he defined the following:
\begin{definition}\label{def:stabinorder}
    Given an ordered family $\mathcal{F}$ of sets, we say that a transversal $\ell$ \textcolor{blue}{stabs the sets in order} if, for every pair of disjoint sets $A$ and $B$ in $\mathcal{F}$ with $A<B$, the line $\ell$ intersects $A$ before $B$.
\end{definition}
In the non-quantitative case, a finite family $\mathcal{F}\subset \mathcal{C}(\mathbb{R}^2)$ has a consistent 1-ordering if and only if there exists some order on the sets in $\mathcal{F}$ such that any three sets have a transversal that stabs them in that order \cite{Wenger.1990,Pollack:1990cna}.
Unlike the non-quantitative case, however, it is possible to have three sets with an $(f,\alpha)$-transversal but no $(f,\alpha)$-consistent ordering.

\begin{example}
\label{ex:three way separated}
Consider the following sets (shown in \cref{fig:three way separated}): Let $A$ be an equilateral triangle of unit area with rightmost edge parallel to the vertical axis. Draw a line $\ell_T$ parallel and slightly to the left of this edge. Let $B$ be the triangle obtained by reflecting $A$ over the line $\ell_T$.
We draw the line $\ell_1$ containing the bottom left corner of $A$ and top corner of $B$, and the line $\ell_2$ containing the top corner of $A$ and bottom corner of $B$.  Each $\ell_i$ splits both $A$ and $B$ into two sets. In each case, the set containing the vertical edge has area $a < 1/2$, and the other has area $1-a$.

\begin{figure}
       \centering
       \includegraphics[width=1\textwidth]
        {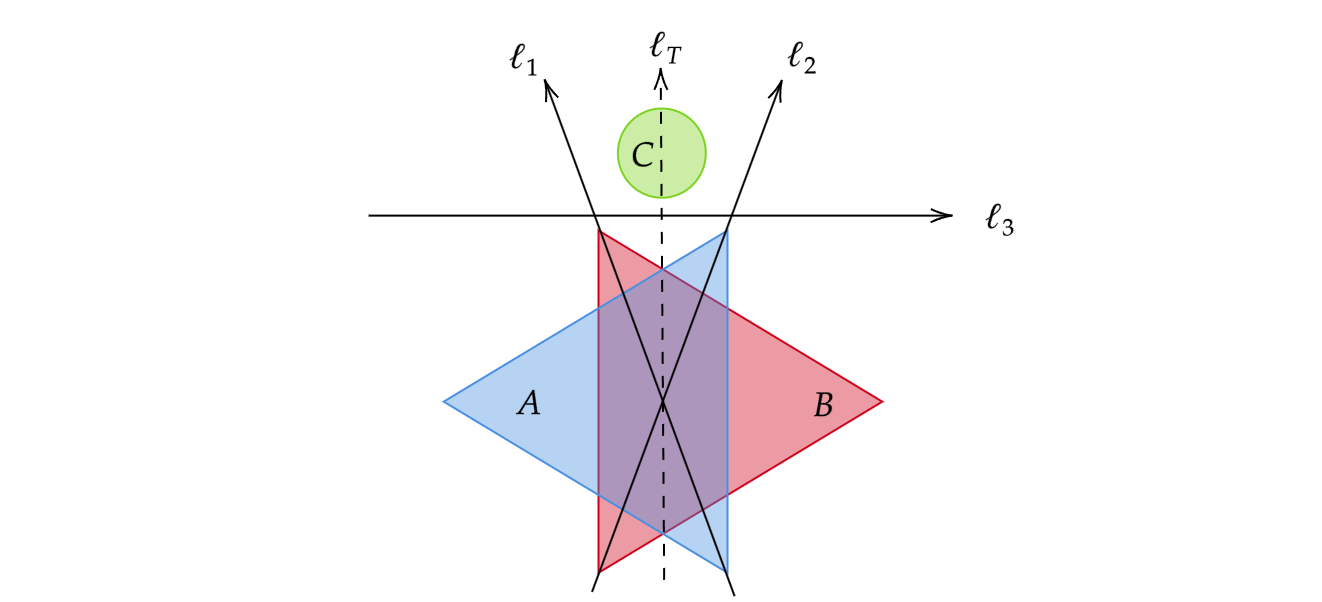}
        \caption{Three sets with an $(f,\alpha)$-transversal $\ell_T$, but no $(f,\alpha)$-consistent ordering, because each set has a line ($\ell_1, \ell_2$ or $\ell_3$) that $(f,\alpha)$-separates it from the other two.} 
        \label{fig:three way separated}
        \end{figure}

 Let $f$ be the area function and  $\alpha = a + \epsilon$ for some sufficiently small epsilon. Then both $\ell_1$ and $\ell_2$ separate the two sets with $A$ on the left and $B$ on the right in both cases. However note that sets on either side of the line $\ell_T$ have area $>\alpha$, so for small enough $\epsilon$, the line $\ell_T$ is an $(f,\alpha)$-transversal of $A$ and $B$.

 Finally, draw any set $C$ above both $A$ and $B$ separated by a line $\ell_3$, and which is $(f,\alpha)$-stabbed by $\ell_{T}$, but not by $\ell_1$ or $\ell_2$. The transversal $\ell_{T}$ stabs all three sets, but since all three sets have a line which $(f,\alpha)$-separates them from the other two, \cref{lem:injectivephi}(b) dictates that there can be no $(f,\alpha)$-consistent ordering of the sets.

\end{example}

 So what went wrong? How can we guarantee three sets with an $(f,\alpha)$-transversal will have an $(f,\alpha)$-consistent ordering? Notice that in \cref{ex:three way separated}, the set $s_{AB}$ is disconnected; the vectors in the directions of lines $\ell_1$ and $\ell_2$ are in two different, disconnected arcs within $s_{AB}$. In general, given any sets $A,B$ with disconnected $s_{AB}$, we can easily find a third set $C$ such that $A,B,C$ have an $(f,\alpha)$-transversal but no $(f,\alpha)$-consistent ordering.

 Thus, when considering the equivalence between an $(f,\alpha)$-consistent ordering and $(f,\alpha)$-transversal,
 it makes sense to restrict ourselves to the case where $s_{AB}$ is a single, connected arc for all pairs $A,B$ in our family $\mathcal{F}$. In this case, we say that a family $\mathcal{F}$ satisfies the \textcolor{blue}{single arc condition}. This gives us a way to define an $(f,\alpha)$-stabbing order on pairs of sets.
 Given a direction $v \in S^1$, let $\textcolor{blue}{\text{Left}(v)} \subseteq S^1$ be the open arc with boundary points $\pm v$ which falls counterclockwise of $v$. Similarly define $\textcolor{blue}{\text{Right}(v)}$.
 
 \begin{definition}
 \label{def:separated stab order}
     Let $A,B \in \mathcal{F}$. 
     Let $\ell$ be an $(f,\alpha)$-transversal in the direction $v \in S^1$. We say $\ell$ \textcolor{blue}{$(f,\alpha)$-stabs the sets in the order} $A<B$ if $s_{AB} \subseteq \text{Left}(v)$.
   Given an ordering on $\mathcal{F}$, we say a transversal $(f,\alpha)$-stabs all of the sets in $\mathcal{F}$ in order if it $(f,\alpha)$-stabs every pair of sets in order. 
    
 \end{definition}

Note that if two sets $A,B$ are not $(f,\alpha)$-separated, i.e., $s_{AB} = \emptyset$, then any $(f,\alpha)$-transversal $(f,\alpha)$-stabs the sets in both order $A<B$ and $B<A$. In the non-quantitative case, this definition is equivalent to regular stabbing order from \cref{def:stabinorder}: For disjoint sets $A,B$, if a transversal $\ell$ (in direction $v$) intersects $A$ before $B$, then all separation directions with $A$ on the left and $B$ on the right lie in $\text{Left}(v)$. If $A$ and $B$ are not disjoint, then a transversal is compatible with either order.

Back to the quantitative case, for $(f,\alpha)$-separated $A,B$, as long as $s_{AB}$ is a single arc, any $(f,\alpha)$-stabbing direction $v \in S^1$ will always fall clockwise of every point in $s_{AB}$ or counterclockwise of every point, while $-v$ does the opposite.
Thus, an \textcolor{blue}{ $(f,\alpha)$-stabbing order} for the pair always exists.
See \cref{fig:stabbing direction} for an example. By contrast, in \cref{fig:three way separated}, the line $\ell_T$ falls clockwise of $\ell_1$ but counterclockwise of $\ell_2$, so it does not stab the sets $A,B$ in either order.

\begin{figure}[h]
    \centering
    \includegraphics[width=1\linewidth]{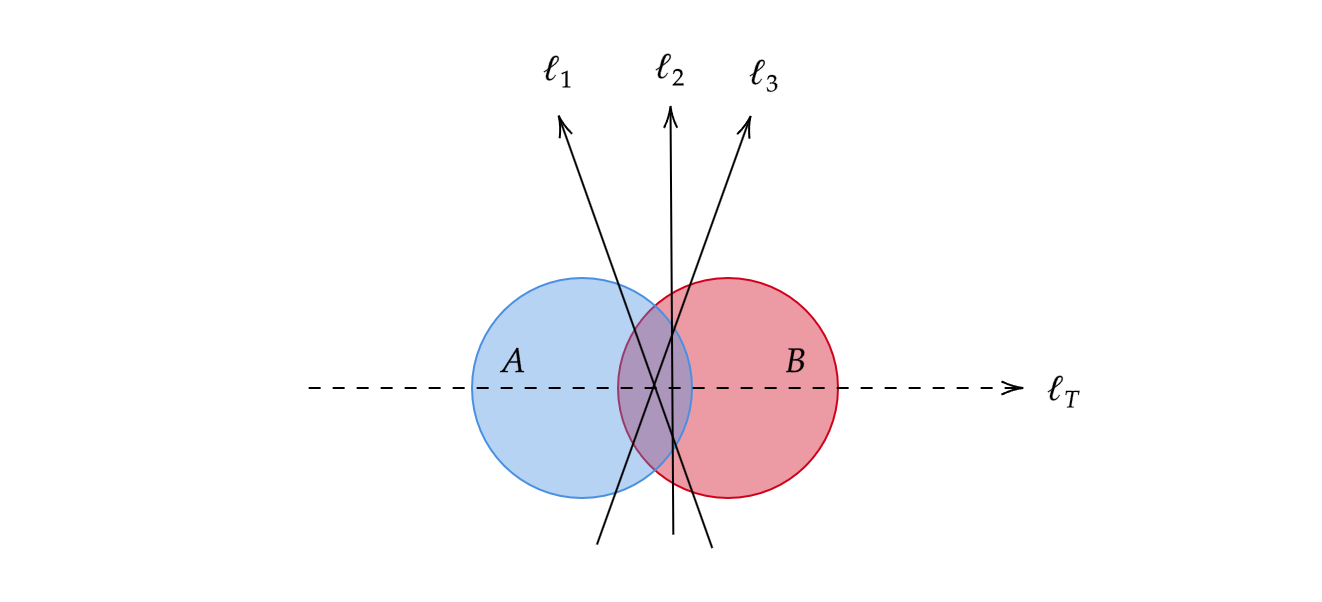}
    \caption{Two sets $A,B$ with $(f,\alpha)$-separation lines $\ell_1,\ell_2,\ell_3 \in s_{AB}$, as well as an $(f,\alpha)$-transversal $\ell_T$ that $(f,\alpha)$-stabs them in the order $A<B$.}
    \label{fig:stabbing direction}
\end{figure}

In addition, if the $(f,\alpha)$-stabbing ordering relation is transitive, then we can extend the pairwise $(f,\alpha)$-stabbing order of an $(f,\alpha)$-transversal to a total order on all the sets it stabs. We state this as a conjecture:

\begin{conjecture}\label{conj:transitiveorder}
    Let $\mathcal{F}$ be a family of sets satisfying the single arc condition.
    Let $A, B, C\in \mathcal{F}$.
    Let $\ell$ be an $(f,\alpha)$-transversal for $\mathcal{F}$ in the direction $v\in S^1$.
    If $s_{AB}, s_{BC}\subseteq \text{Left}(v)$, then $s_{AC}\subseteq \text{Left}(v)$.
\end{conjecture}
\noindent
We say that a family $\mathcal{F}$ satisfies the \textcolor{blue}{transitivity condition} if \cref{conj:transitiveorder} holds. For the purpose of this paper, when we care about order, it will be enough to require that a transversal $(f,\alpha)$-stabs every pair of sets in agreement with a preexisting order on $\mathcal{F}$ without needing the result of \cref{conj:transitiveorder}. Additionally, when all sets in $\mathcal{F}$ are disjoint, the transitivity condition holds.

\begin{proposition}\label{prop:disjoint-singlearc}
Let $\mathcal{F}$ be a family of pairwise disjoint sets satisfying \cref{rmk:all directions stab}. 
    Then, $\mathcal{F}$ satisfies the single arc condition. Furthermore, the $(f,\alpha)$-stabbing order of a transversal $\ell$ is given by the order in which $\ell$ intersects the sets,
so the $(f,\alpha)$-stabbing order is transitive.
\end{proposition}
\begin{proof}
    Let $A, B\in\mathcal{F}$.
    As $A, B$ are convex and disjoint, there exists a directed line $\ell$ disjoint from $A$ and $B$ such that $A$ lies fully to the left of $\ell$ and $B$ to the right.
    Let $u\in s_{AB}$ be the pointing direction of $\ell$. To prove $s_{AB}$ is a single arc, we want to show that for any two points $v_1, v_2 \in s_{AB}$, there exists an arc connecting them which lies entirely in $s_{AB}$.
    By \cref{remark:sAB-sBA}, the set $s_{AB}$ cannot be the whole circle.
    Thus there exists a point $x\in S^1$ such that $x\notin s_{AB}.$
   Let $\overset{\Huge\frown}{v_1v_2}$ be the closed arc with endpoints $v_1, v_2$ which does not contain $x$. Define arcs $\overset{\Huge\frown}{v_1u}$ and $\overset{\Huge\frown}{v_2u}$ similarly. We will show that $\overset{\Huge\frown}{v_i u} \subseteq s_{AB}$ for $i=1,2$, therefore implying $\overset{\Huge\frown}{v_1v_2} \subseteq \overset{\Huge\frown}{v_1u} \cup \overset{\Huge\frown}{v_2u} \subseteq s_{AB}$.

    For $i= 1,2$, fix $(f,\alpha)$-separation line $\ell_i$ in the direction $v_i$. 
    Let $p_i = \ell_i\cap \ell$ (See \cref{fig:disjoint-singlearc} for example layout). Note that $p_1, p_2\notin A\cup B$.  Rotate $\ell_i$ around the point $p_i$ until it agrees with $\ell$, keeping its direction in $\overset{\Huge\frown}{v_iu}$. Assume without loss of generality that we rotate counterclockwise. Each directed line $\ell'$ that we rotate through along the way remains $(f,\alpha)$-separating because $\ell^+(B)\subseteq \ell_i'^+(B)\subseteq 
        \ell^+_i(B)$ with $f(\ell^+_i(B))<\alpha$
        and $\ell^-(A)\subseteq \ell_i'^-(A)\subseteq \ell_i^-(A)$ with $f(\ell_i^-(A))<\alpha$. Thus $\overset{\Huge\frown}{v_i u} \subseteq s_{AB}$ for $i=1,2$.

    Any $(f,\alpha)$-transversal of $A,B$ that intersects $A$ first in this case has a direction $v$ satisfying $s_{AB} \subseteq \text{Left}(v)$, so the $(f,\alpha)$-stabbing order satisfies $A < B$.  
    \begin{figure}[h]
       \begin{center}
       \includegraphics[width=1\textwidth]
        {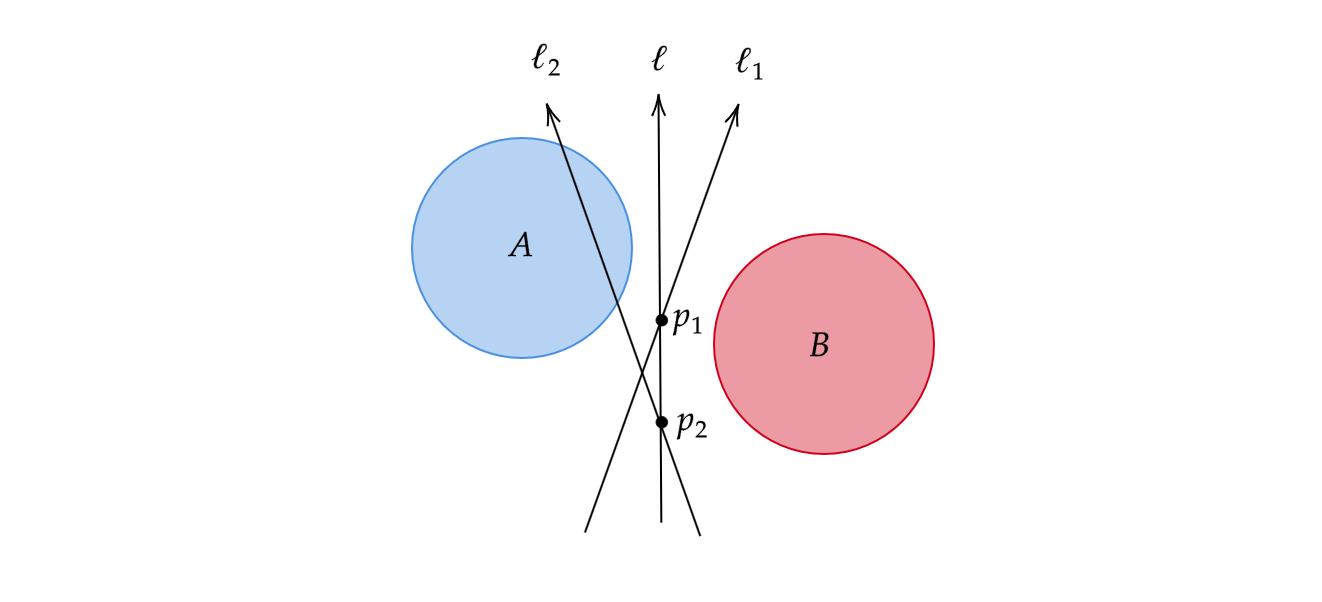}
        \caption{\label{fig:disjoint-singlearc}Given $\ell_1$ in the direction $v_1$ and $\ell_2$ in the direction $v_2$, any direction $v\in\overset{\Huge\frown}{v_1v_2}$ is also in $s_{AB}$. }
        \end{center}
        \end{figure}
\end{proof}

\begin{remark}
\label{rmk:not convex weird}
    If we define monotone functions appropriately for compact, non-convex sets (see \cref{remark:convex-connected}), then the single arc condition from \cref{prop:disjoint-singlearc} works for non-convex sets as well when we strengthen the disjointness condition to require that every pair of sets has a (non-quantitative) separation line. If additionally we require that the sets are connected, then the rest of \cref{prop:disjoint-singlearc} (as well as \cref{cor:disjoint} below)
    works as well for non-convex sets. However, if a set is disconnected, an $(f,\alpha)$-stabber may not intersect the set at all.
\end{remark}

For non-disjoint sets, the $(f,\alpha)$-stabbing order of a transversal may be different from the order in which it intersects the sets. 
(See \cref{nontransversal counterexample}.)
\begin{figure}[h]
       \begin{center}
       \includegraphics[width=.8\textwidth]
        {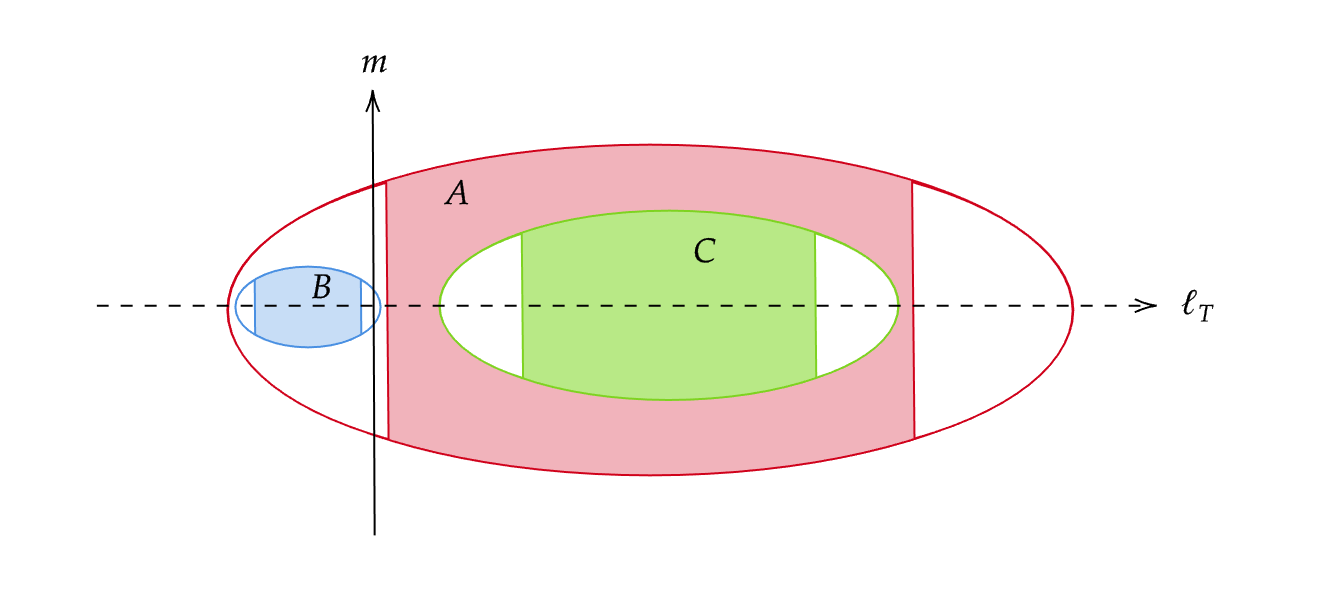}
        \caption{\label{nontransversal counterexample}Non-disjoint sets $A$, $B$ and $C$ with an $(f, \alpha)$-stabbing line $\ell_T$ and $(f,\alpha)$-separation line $m$ showing that $s_{BA}, s_{BC}\subseteq \text{Left}(\ell_T)$. Note that line $\ell_T$ intersects $A$ before $B$, but $(f,\alpha)$-stabs $B$ before both $A$ and $C$ in $(f,\alpha)$-stabbing order.}
        \end{center}
        \end{figure}

In the next few propositions, we show that for a family $\mathcal{F}$ satisfying \cref{rmk:all directions stab}, the single arc condition, and the transitivity condition, the existence of an $(f,\alpha)$-consistent ordering of $\mathcal{F}$ is equivalent to the existence of some ordering of $\mathcal{F}$ so that every three sets in $\mathcal{F}$ has an $(f,\alpha)$-transversal that $(f,\alpha)$-stabs the sets in order.

\begin{proposition}\label{prop:single+stab->consord}
Let $\mathcal{F}$ be an ordered family of sets satisfying \cref{rmk:all directions stab} and the single arc condition.
Assume that every three sets in $\mathcal{F}$ have a transversal that $(f,\alpha)$-stabs them in order. Then, $\mathcal{F}$ has an $(f,\alpha)$-consistent ordering given by the prescribed ordering of $\mathcal{F}$.
\end{proposition}

\begin{proof}
    Let $\mathcal{F}=\{C_1,\ldots,C_n\}$ with order $C_1<\ldots<C_n$. Suppose every three sets have a transversal that $(f,\alpha)$-stabs them in order. Let $\phi: \mathcal{F} \rightarrow \mathbb{R}$ be any function with $\phi(C_1)<\ldots<\phi(C_n)$.
   Towards contradiction, assume $\phi$ is not an $(f,\alpha)$-consistent ordering. Then, we can fix $A<B<C \in \mathcal{F}$ and a line $\ell$ which $(f,\alpha)$-separates the sets with $B$ on the left and $A,C$ on the right. Note that the line $\ell$ has pointing direction $v \in s_{BA} \cap s_{BC}$.
    However, by assumption, there is some line $\ell'$ that stabs sets $A,B,C$ in order, so by letting $v'\in S^1$ be the direction of $\ell'$, we have $s_{AB}, s_{BC} \subseteq \text{Left}(v')$. Thus $s_{BA} \subseteq \text{Right}(v')$, so $s_{BA}\cap s_{BC}=\emptyset$, a contradiction.
\end{proof}

\begin{proposition}\label{prop:single+trans->stab}
    Let $\mathcal{F}$ be a family of sets satisfying \cref{rmk:all directions stab}, the single arc condition, and the transitivity condition.
    Assume that $\mathcal{F}$ has an $(f,\alpha)$-consistent ordering. Then, there is an ordering of $\mathcal{F}$ so that any three sets have a transversal that $(f,\alpha)$-stabs the sets in order.
\end{proposition}

\begin{proof}
    By assumption, fix an $(f,\alpha)$-consistent ordering $\phi:\mathcal{F}\to \mathbb{R}$.
    Define an ordering on the sets in $\mathcal{F}$ by letting $A\leq B$ iff $\phi(A)\leq \phi(B)$. If $\phi(A)=\phi(B)$, the sets are not $(f,\alpha)$-separated, and we can order them arbitrarily to obtain a total order on sets in $\mathcal{F}$.
    Pick any $A < B < C\in \mathcal{F}$. 
    By \cref{lem:injectivephi} (b), there is no $(f,\alpha)$-separation line with $B$ on one side and $A,C$ on the other. Thus, $s_{AB}\cap s_{CB}=s_{BA}\cap s_{BC}=\emptyset$.
     
    Towards contradiction, assume there is no transversal that $(f,\alpha)$-stabs $A,B,C$ in order $A < B < C$. Thus, the transitivity condition implies that no transversal can have both $s_{AB}$ and $s_{BC}$ contained on the left side of its pointing direction. 
    By \cref{thrm:quantitativehadwiger}, fix an $(f,\alpha)$-transversal $\ell$ for $\mathcal{F}$.
    Let $v\in S^1$ be the direction of $\ell$.
    As $\ell$ cannot $(f,\alpha)$-stab $A,B,C$ in order $A < B < C$, we have either $s_{AB}, s_{CB}\subseteq \text{Left}(v)$ or $s_{BA},s_{BC}\subseteq \text{Left}(v)$. 
    By symmetry, we can assume $s_{AB}, s_{CB}\subseteq \text{Left}(v)$ with $s_{AB}$ falling closer to $v$ than $s_{CB}$ does (pictured in \cref{fig:sAB-sCB-u}). 
    Let $U$ be the maximal closed arc between $s_{BA}$ and $s_{BC}$ which is fully contained in $\text{Right}(v)$. Since $s_{BA}\cap s_{BC}=\emptyset$, we know $U$ is nonempty. For all $u\in U$, we have $s_{AB}, s_{BC}\subseteq \text{Left}(u)$.
    Since there is no transversal that $(f,\alpha)$-stabs $A,B,C$ in order $A < B < C$, every $u\in U$ must be in $s_{AC}$ or $s_{CA}$. 
    Since $s_{AC}$ and $s_{CA}$ are both single arcs which are non-overlaping (\cref{remark:sAB-sBA}) and open (\cref{lem:open})
    we must have $U\subseteq s_{CA}$ or $U\subseteq s_{AC}.$
    If $U\subseteq s_{AC}$, the arc $s_{AC}$ must extend beyond the endpoints of $U$ on both sides, so $s_{BA}\cap s_{AC}\neq \emptyset$ with $s_{BA}\cap s_{AC}\subseteq \text{Right}(u)$ for each $u\in U$. However, any vector in $s_{BA}\cap s_{AC}$ is also in $s_{BC}$, so it contradicts that $s_{BC}\subseteq \text{Left}(u)$.
    The contradiction follows similarly when $U\subseteq s_{CA}$.
    \begin{figure}[h]
       \begin{center}
       \includegraphics[width=\textwidth]
        {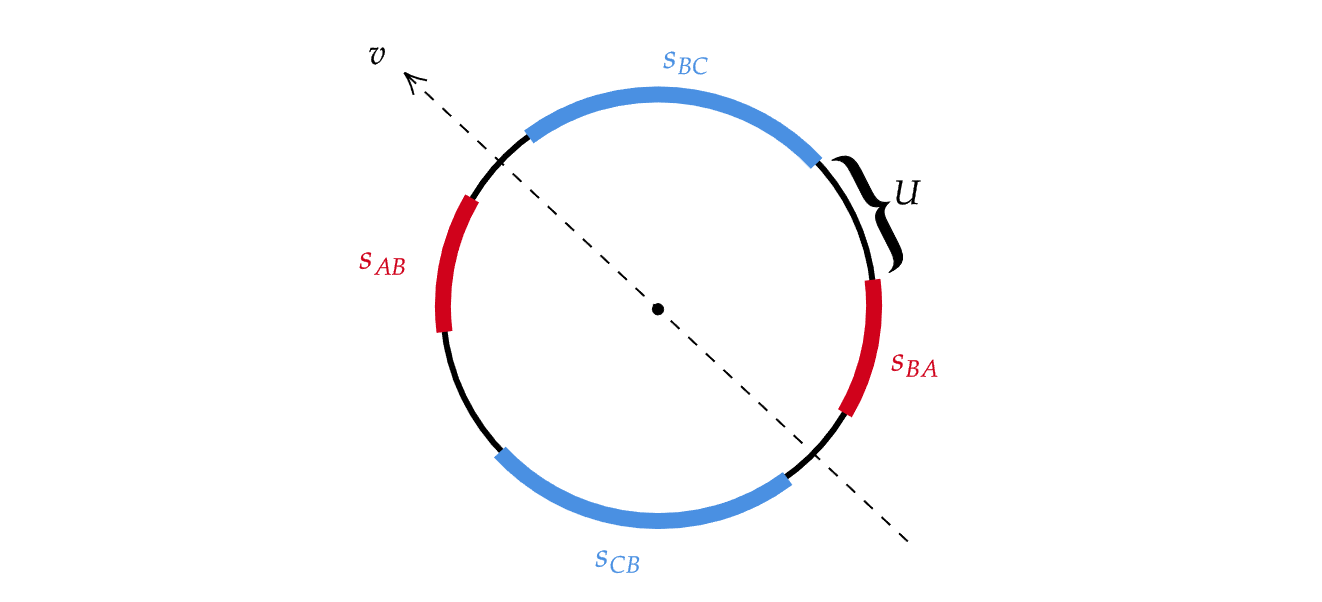}
        \caption{\label{fig:sAB-sCB-u}A direction $v\in S^1$ with $s_{AB}, s_{CB}\subseteq \text{Left}(v)$, and the closed arc $U\subseteq S^1$ with $s_{AB},s_{BC}\subseteq \text{Left}(u)$ for each $u\in U$. Note that the relative location of $s_{AB}, s_{CB}\subseteq \text{Left}(v)$ could have been flipped, but the result would be equivalent by symmetry. }
        \end{center}
        \end{figure}
\end{proof}

\begin{corollary}
\label{cor:disjoint}
  Let $\mathcal{F}$ be a family of disjoint sets satisfying \cref{rmk:all directions stab}. Then $\mathcal{F}$ has an $(f,\alpha)$-consistent ordering if and only if there exists an ordering on the sets such that every three sets have a transversal which $(f,\alpha)$-stabs the sets in order.
\end{corollary}

\begin{proof}
By \cref{prop:disjoint-singlearc}, $\mathcal{F}$ satisfies the single arc condition and the transitivity condition. 
    Therefore, the if direction follows from \cref{prop:single+stab->consord}, and the only if direction follows from \cref{prop:single+trans->stab}.  
\end{proof}

\subsection{Quantitative Hadwiger variants}

In terms of stabbing order, Hadwiger's theorem and our \cref{thrm:quantitativehadwiger} have provided conditions for when some $(f,\alpha)$-transversal exists, but do not guarantee that the transversal will intersect all of our sets in the order provided.
We now give quantitative versions of \cite[Theorems 2 and 3]{Wenger.1990} which give conditions for when we can find a transversal that $(f,\alpha)$-stabs all the sets in $\mathcal{F}$ in order. However, in order to guarantee we are in a setting where the $(f,\alpha)$-stabbing order is well defined, we will require that $\mathcal{F}$ satisfies the single arc condition. This will also be important later in the proof for Wenger's method to work in this setting.

\begin{theorem}
\label{thrm:6stabber}
Let $\mathcal{F} = \{C_1,\ldots, C_n\}$ ($n\geq 6$) be an ordered family of sets satisfying \cref{rmk:all directions stab} and the single arc condition. Suppose that for any six sets, there exists a transversal that $(f,\alpha)$-stabs all six of them in the order provided. Then, there exists a transversal that $(f,\alpha)$-stabs sets in $\mathcal{F}$ in the order provided. 
\end{theorem}

The proof of \cref{thrm:6stabber} follows directly from Wenger's original proof in \cite[Theorem 2]{Wenger.1990} when we adapt to the quantitative definitions and lemmas described in this paper.

\begin{proof}[Proof of \cref{thrm:6stabber}] 
Assume that $\mathcal{F}$ is ordered as $C_1<\cdots<C_n$.
For a pair of sets $C_i, C_j$, let $t_{C_iC_j} \subseteq S^1$ be the set of all directions of transversals which $(f,\alpha)$-stab the sets in the order $C_i<C_j$. For every pair $C_i, C_j$, the set $t_{C_iC_j}$ is nonempty. Because $s_{C_iC_j}$ is a single \emph{open} arc, by \cref{def:separated stab order}, $t_{C_iC_j}$ is a single \emph{closed} arc. Further, each $t_{C_iC_j}$ is either all of $S^1$ or has measure less than $180^{\circ}$. Every three arcs in $\{t_{C_iC_j}|i<j\}$ must intersect since we know there is a transversal that $(f,\alpha)$-stabs all six sets involved in that order simultaneously. Thus all the arcs in $\{t_{C_iC_j}|i<j\}$ share a common point. 
\cref{1dhellyuse} then gives an $(f,\alpha)$-transversal for $\mathcal{F}$ in the direction of the common point. This transversal $(f,\alpha)$-stabs each pair of sets in order.
\end{proof}

The above proof requires that every \emph{six} sets have a transversal that $(f,\alpha)$-stabs them in order. In Wenger's non-quantitative version of \cref{thrm:4stabber} (\cite[Theorem 3]{Wenger.1990}), he was able to decrease the number from six to four by adding the requirement that the sets be disjoint. We can do similarly here, but now the requirement can be weakened to the quantitative condition that all sets are $(f,\alpha)$-separated rather than necessarily disjoint.

\begin{theorem}
\label{thrm:4stabber}
Let $\mathcal{F} = \{C_1,\ldots, C_n\}$ ($n\geq 4$) be an ordered family of sets satisfying \cref{rmk:all directions stab}. Assume $s_{C_iC_j}$ is a \emph{nonempty} single arc for all $C_i,C_j \in \mathcal{F}$. Suppose that for any four sets, there exists a transversal that $(f,\alpha)$-stabs all four of them in the order provided. Then, there exists a transversal that $(f,\alpha)$-stabs all sets in $\mathcal{F}$ in the ordering provided.
\end{theorem}

\begin{proof}
Similarly to the proof of \cref{thrm:6stabber}, this proof follows directly from Wenger's original proof of \cite[Theorem 3]{Wenger.1990} when we adapt to our quantitative definitions and lemmas. 

We begin by proving \cref{thrm:4stabber} in the case where our family has exactly five sets, i.e. $\mathcal{F}=\{C_1,C_2,C_3,C_4,C_5\}$ ordered by the numbers. Any four of these sets have a transversal which $(f,\alpha)$-stabs them in order. Let $P_{1234}$ be the point on the unit circle corresponding to this $(f,\alpha)$-stabbing direction of $\{C_1,C_2,C_3,C_4\}$. Similarly name such points for all other choices of four sets. Given any three such points, say $P_{1234}, P_{1235}, P_{1245}$ without loss of generality, they all mutually share two sets in common, in this case $C_1$ and $C_2$. Thus all three points fall inside $t_{C_1 C_2}$ (as defined in the proof of \cref{thrm:6stabber}) which we again know is a single arc. Since $s_{C_1 C_2}$ is non-empty, $t_{C_1 C_2}$ cannot cover more than half the circle. Thus $P_{1234}, P_{1235}, P_{1245}$ all lie within some half circle. Suppose that the convex hull of all five points contained the center of the circle. Then by Carath\'eodory's theorem, the center of the circle lies within the convex hull of three of these five points. However, this is not possible since any three points fall within a half circle. Thus, all five points must fall within some half circle, $\theta$.

Let
$S := \{t_{C_iC_j}| 1\leq i < j \leq 5\}$. Intersect all arcs in $S$ with $\theta$ to get a set of arcs $S'$. Any pair of arcs overlap at one of the five points. Thus by Helly's theorem, the intersection of all the arcs is nonempty. This gives us a direction which is an $(f,\alpha)$-stabbing direction for all pairs in numerical order, so by \cref{1dhellyuse} there is an $(f, \alpha)$-stabber for all the sets in order.

Now suppose we have more sets in our family. We will show any six of them have a transversal which $(f,\alpha)$-stabs them in order. Then by \cref{thrm:6stabber}, there is an $(f,\alpha)$-stabber of all our sets in order. Given any six sets $\{C_1,C_2,C_3,C_4,C_5, C_6\}$, we have a transversal for any five of them which $(f,\alpha)$-stabs them in order. This gives us six points $P_{12345}, P_{12346}, P_{12356}, P_{12456}, P_{13456}, P_{23456}$ on the unit circle corresponding to those $(f, \alpha)$-stabbing directions. By the same reasoning as before, all six points lie in some half circle $\theta$. Again let $S'$ be the set of arcs obtained by intersecting $\theta$ with $S$ (defined as before, now on arcs for all sets in $\{C_1, \dots , C_6\}$). As before, all arcs share a point, which results in a transversal that $(f,\alpha)$-stabs all of the sets in order.
\end{proof}

\section{Colorful generalization}
\label{colorful}

We now work on generalizing our results to a colorful setting. In 2008, Arocha, Bracho, and Montejano \cite{Arocha:2008ti} proved the following colorful generalization for Hadwiger's theorem:

\begin{theorem}[\cite{Arocha:2008ti}]\label{thm:colorhadwiger}
    Given a finite, ordered family of compact, convex sets in the plane colored red, green, or blue; if for every triple $C_i,C_j,C_k$ of differently colored sets there is some traversal that stabs them in order, then there exists a color for which there is a transversal for all the sets of that color in the family.
\end{theorem} 

Similarly to Wenger's method, the method used in their paper does not require disjointness for the result to hold. In this section, we generalize their method to obtain \cref{thm:colorfulquanthadwiger}, a colorful generalization to \cref{thrm:quantitativehadwiger}. 
\begin{assumption}\label{assump3}
    For the rest of this section, assume that in addition to a family $\mathcal{F}$ satisfying \cref{rmk:all directions stab}, each set in $\mathcal{F}$ is assigned a color red, green, or blue. 
\end{assumption}
Note that if $\mathcal{F}$ has no sets in one of the three colors, then trivially any line is an $(f,\alpha)$-transversal for all sets of that color. We call a subfamily of $\mathcal{F}$ consisting of all the sets of a given color a \textcolor{blue}{monochromatic subfamily}. Now, here are colorful analogs of \cref{def:consistent-ordering} and \cref{lem:injectivephi} (b).
\begin{definition}\label{defn:colorful_consistent_ordering}
We say that the family $\mathcal{F}$ has a \textcolor{blue}{colorful $(f,\alpha)$-consistent ordering} if there exists a map $\phi: \mathcal{F}\to \mathbb{R}$ such that for any subfamilies $\mathcal{F}_1, \mathcal{F}_2 \subseteq \mathcal{F}$ with $|\mathcal{F}_1|+|\mathcal{F}_2|\leq 3$ and $\mathcal{F}_1\cup\mathcal{F}_2$ containing no two sets of the same color,
  if there is an $(f,\alpha)$-separation line with sets in $\mathcal{F}_1$ on one side and $\mathcal{F}_2$ on the other, then
  $conv(\phi(\mathcal{F}_1)) \cap conv(\phi(\mathcal{F}_2)) = \emptyset.$

\end{definition}

\begin{lemma}\label{lem:threesetscolorful}
    Suppose the family $\mathcal{F}$ has a colorful $(f,\alpha)$-consistent ordering with $\phi:\mathcal{F}\to \mathbb{R}$. Let $A,B, C\in \mathcal{F}$, no two of the same color, such that $\phi(A) \leq \phi(B) \leq \phi(C)$. Then there is no $(f,\alpha)$-separation line with $B$ on one side and $A,C$ on the other.
\end{lemma}

From a colorful $(f,\alpha)$-consistent ordering on $\mathcal{F},$ we obtain a linear ordering on the family. This is done by simply taking the ordering on the image of $\phi.$ It is possible two different sets have the same image if they are of the same color, or if they are of different colors but are not $(f,\alpha)$-separated. In either of these cases, 
arbitrarily assign a relative order on these sets. Thus, we can assume our family $\mathcal{F}$ with a colorful $(f,\alpha)$-consistent ordering comes with a numerical order $C_1, C_2 , \dots, C_n$ compatible with $\phi$.

The approach for proving \cref{thm:colorfulquanthadwiger} is similar to that of \cref{thrm:quantitativehadwiger}. 
We define two subsets $S_1$ and $S_{-1}$ of the unit circle, and show that if no monochromatic subfamily has an $(f,\alpha)$-transversal, then $S_1$ and $S_{-1}$ cover the whole circle. 
We then show that if there exists a colorful $(f,\alpha)$-consistent ordering but no monochromatic subfamily has an $(f,\alpha)$-transversal, then $S_1$ and $S_{-1}$ cannot cover the whole circle, producing a contradiction. 
 To begin, we establish some key definitions which will allow us to define our sets $S_1$ and $S_{-1}$.

Fix an origin in $\mathbb{R}^2$.
For a direction $v\in S^1$, let $\ell_v(0)$ be the line through the origin in the direction of $v$. Let $u$ be the unit vector pointing in the direction 90 degrees clockwise of $v$.
Every other line in the direction $v$ is of the form $\ell_v(c):=\ell_v(0)+c\cdot u$ for some constant $c\in \mathbb{R}$, so we can identify these lines with the real numbers. For a given $C_i\in \mathcal{F},$ let $a_i=\min_{c\in\mathbb{R}}\{C_i \text{ is } (f,\alpha)\text{-stabbed by } \ell_v(c)\}$ and $b_i=\max_{c\in\mathbb{R}}\{C_i \text{ is } (f,\alpha)\text{-stabbed by } \ell_v(c)\}$. (Note we required an $(f,\alpha)$-stabber to exist in each direction by \cref{rmk:all directions stab}, and from there a unique minimum/maximum value always exists since the set of $c$ that makes $\ell_v(c)$ stabbing is a closed set.) Now, considering only the red sets, let $p_{red}=\max_{i\text{ red}}\{a_i\}$ and $q_{red}=\min_{i\text{ red}}\{b_i\}$. Similarly, we define $p_{green},q_{green},p_{blue},q_{blue}.$ Order the set $\{p_{red},p_{green},p_{blue}\}$ from smallest to largest and let $p_{m}$ be the median one. Similarly define $q_{m}.$ (See \cref{fig:lemma4.3}.)

\begin{figure}
    \centering
    \includegraphics[scale=0.25]{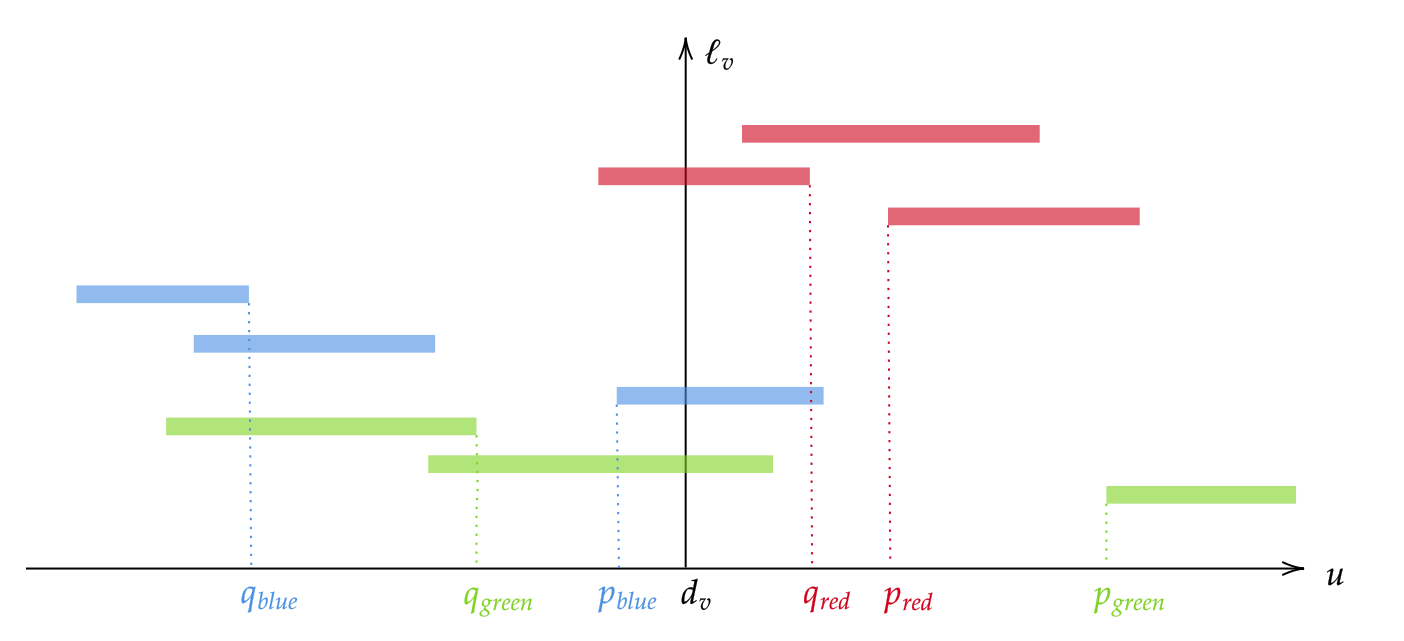}
    \caption{An example of $p_{blue},p_{red},p_{green},q_{blue},q_{red},q_{green}$ and $\ell_v$ for some colorful intervals. In this case, $p_m=p_{red}$ and $q_m=q_{green}$.} 
    \label{fig:lemma4.3}
\end{figure}

\begin{definition}\label{def:middle-f-alpha}
In the context of the previous paragraph, fix $d_v=\frac{p_{m}+q_{m}}{2}.$ We
call the line $\ell_v:=\ell_v(d_v) = \ell_v(0)+d_v\cdot{u}$ the \textcolor{blue}{middle $(f,\alpha)$-line} of $\mathcal{F}$ in the direction $v$ (see \cref{fig:lemma4.3}). We define the \textcolor{blue}{middle $(f,\alpha)$-vector} for $v$ to be the colorful vector $(x_1,\cdots,x_n),$ where $x_i$ is 0 if the line $\ell_{v}$ is an $(f,\alpha)$-stabber of $C_i,$ is $-1$ if $f(\ell_{v}^-(C_i))<\alpha,$ and $1$ if 
$f(\ell_{v}^+(C_i))<\alpha.$ The color of each coordinate is the same as the one for the respective set.
\end{definition}

We use the reference line $\ell_v(0)$ so that the definitions are easier to construct, but note that the middle $(f,\alpha)$-line for $v$ does not depend on the choice of the reference line or origin. If we substitute $\ell_v(0)$ with some other line in the direction $v$ a distance $c$ from it, all the $p$'s and $q$'s will change by $-c,$ so $\ell_v$ stays the same. Our assignment of middle $(f,\alpha)$-lines is continuous with respect to the direction $v \in S^1$.

\begin{lemma}\label{lem:continuous}
    Let $\mathcal{B}$ be a closed ball around the origin containing all sets in $\mathcal{F}$.
    Then for $v \in S^1$, the map $v\mapsto \ell_v \cap \mathcal{B}$ is continuous with respect to Hausdorff distance.
\end{lemma}

\begin{proof}
    Let $\mathcal{F}=\{C_1,\ldots,C_n\}.$
    It is enough to show that the choice of $d_v$ as in \cref{def:middle-f-alpha} is continuous with respect to the choice of $v \in S^1$. Recall our reference lines $\ell_v(0)$ are obtained from each other by rotating about the origin. Since the function $f$ is continuous (see \cref{remark:hasudorffmetric}), by rotating $\ell$ continuously around the origin, the values of $a_i$ and $b_i$ will also change continuously for every $i$ ($1\leq i\leq n$). Thus the $p$'s, $q$'s, and $d_v$'s will also change continuously. 
\end{proof}

Now, we define our subsets $S_1$ and $S_{-1}$ of the circle. We let $S_1$ be the subset corresponding to all directions whose middle $(f,\alpha)$-vector has a $1$ as its first nonzero entry, and similarly let $S_{-1}$ be the subset corresponding to all directions whose middle $(f,\alpha)$-vector has a $-1$ as its first nonzero entry. 

Note that if any point on the circle is \emph{not} in $S_1 \cup S_{-1}$, then the middle $(f,\alpha)$-line in that direction is an $(f,\alpha)$-transversal of not just one monochromatic subfamily but of all sets of all colors at once. Thus, if $\mathcal{F}$ has no monochromatic subfamily with a common $(f,\alpha)$-transversal,
then $S_1 \cup S_{-1}$ must cover the whole circle. As vectors for opposite directions have opposite signs for their first nonzero entries, both $S_1$ and $S_{-1}$ must be nonempty. So to produce a contradiction we just need to show that if $\mathcal{F}$ has no monochromatic subfamily with a common $(f,\alpha)$-transversal, then $S_1$ and $S_{-1}$ must be open (and they are clearly disjoint), meaning that they \emph{cannot} cover the entire circle.

\begin{definition}
A colorful sign vector is said to be \textcolor{blue}{balanced} if for every two colors there exists a coordinate in one of the colors with a 1 in it, and a coordinate in the other color with a $-1$ in it. That is, for any pair of colors, we can find two sets, one of each color, that ``lie'' on opposite sides of the line. 
\end{definition}

\begin{lemma}\label{lem:balanced}
If a colorful family $\mathcal{F}$ has no monochromatic subfamily with a common $(f,\alpha)$-transversal in the direction $v$, then the middle $(f,\alpha)$-vector of $\mathcal{F}$ in the direction $v$ is balanced.
\end{lemma}
\begin{proof}
This proof is adapted from \cite[Lemma 3]{Arocha:2008ti} using the quantitative $p_{red},q_{red},p_{blue}, q_{blue}, p_{green}, q_{green}, p_{m},q_{m}$ as defined above, and $d_v= \frac{p_{m}+q_{m}}{2}$.
By assumption, there is no $(f,\alpha)$-transversal to the red sets in the direction $v$, so we have $p_{red}>q_{red}$. Similarly, $p_{green}>q_{green}$ and $p_{blue}>q_{blue}$. After ordering, we get $p_m>q_m$. 

Now, let $x$ be the middle $(f, \alpha)$-vector in the direction $v$, and suppose $x$ is not balanced. This means there are two colors (without loss of generality red and blue) such that every non-zero red and blue coordinate of $x$ have the same sign (say it is a 1), or the entries are all zero. Then, both $q_{red}$ and $q_{blue}$ are greater than or equal to $d_v$, since all the sets of these colors lie on the positive side of the line $\ell_v$. Since we also have $p_{m}>q_{m}$, and $d_v$ is in between $p_{m}$ and $q_{m}$, we get $d_v>q_{m}.$ But then both $q_{red}$ and $q_{blue}$ are larger than $q_{m}$, which is impossible since $q_{m}$ is smaller than exactly one of the $q$'s. Therefore, we have a contradiction, so the vector $x$ has to be balanced.
\end{proof}

Next, we show that when $\mathcal{F}$ has a colorful $(f,\alpha)$-consistent ordering, the middle $(f,\alpha)$-vectors satisfy another property called being Hadwiger. We will then get a contradiction by showing that when all our middle $(f,\alpha)$-vectors are both balanced and Hadwiger, the sets $S_1$ and $S_{-1}$ must be open.

\begin{definition}
A colorful sign vector $x$ is \textcolor{blue}{Hadwiger} if whenever we have $i<j<k$, and $x_i,x_j,x_k$ are of different colors, then $(x_i,x_j,x_k)$ is not equal to $(-1,1,-1)$ or $(1,-1,1).$
\end{definition}

Intuitively, if a middle $(f,\alpha)$-vector is Hadwiger, then the corresponding middle $(f,\alpha)$-line does not $(f,\alpha)$-separate any three differently colored sets with the middle set on one side and the first and last sets on the other side.

\begin{lemma}\label{lem:corlofuldisjoint}
If a family $\mathcal{F}$ satisfies the conditions for \cref{thm:colorfulquanthadwiger}, then for any direction $v$, the corresponding middle $(f,\alpha)$-vector is Hadwiger.
\end{lemma}
\begin{proof}
This is a direct consequence of \cref{lem:threesetscolorful}. If there was a direction such that the middle $(f,\alpha)$-vector in that direction was not Hadwiger, we would have three sets $A, B, C$, no two of the same color, with $\phi(A)\leq \phi(B)\leq \phi(C)$ and an $(f,\alpha)$-separation line with $B$ on one side and $A, C$ on the other.
\end{proof}

Now, consider a partial ordering on the middle $(f, \alpha)$-vectors which relates vectors for middle $(f, \alpha)$-lines that separate sets in similar ways as each other. Given two $n$-dimensional vectors, we say $(x_1,x_2,\cdots,x_n)\preceq (y_1,y_2,\cdots,y_n)$ if and only if $x_i\neq 0$ implies $x_i=y_i$ for every $i.$ For middle $(f, \alpha)$-vectors in particular, it tells us the following:
Let $x,y$ be middle $(f,\alpha)$-vectors of some family $\mathcal{F}$, obtained from the middle $(f,\alpha)$-lines $\ell_u$ and $\ell_{v}$ in directions $u, v$ respectively. If $x \preceq y$, then 
$\ell_u$ may $(f,\alpha)$-stab sets which $\ell_v$ does not, but otherwise all sets fall to the same side of $\ell_u$ as of $\ell_v$.  With this partial order, and the following lemma from \cite{Arocha:2008ti}, we can finish the proof of \cref{thm:colorfulquanthadwiger}.

\begin{lemma}[\cite{Arocha:2008ti}, Lemma 4]
\label{lem:hadwigerbalanced}
If $x,y$ with $x\preceq y$ are both balanced and Hadwiger colored sign vectors,
then their first non-zero entry is the same (or all of the entries are zero).
\end{lemma}

\begin{proof}[Proof of Theorem 1.2]
Recall we have defined the sets $S_1$ and $S_{-1}$, and shown already that if there is no monochromatic subfamily which has an $(f,\alpha)$-transversal to all sets of that color, then $S_1$ and $S_{-1}$ cover the whole circle, are disjoint, and that $S_1, S_{-1} \neq \emptyset$. To finish the proof, we will provide a contradiction by showing that if additionally the conditions of \cref{thm:colorfulquanthadwiger} are met, then the sets $S_1$ and $S_{-1}$ are open. In this case, since the circle is connected, they cannot possibly cover the whole circle which gives us our contradiction.

To prove $S_1$ is open, we need to show that for any direction $v$ in $S_1$, there exists some neighborhood $U$ of $v$ which is fully contained in $S_1$ (and likewise for $S_{-1}$). In terms of middle $(f,\alpha)$-lines, this means that we need some neighborhood $U$ of $v$ such that for all $v' \in U$, the middle $(f,\alpha)$-vector in the direction of $v'$ has first nonzero entry equal to 1 (or $-1$ in the case of $S_{-1}$). 

First, for every direction, the middle $(f, \alpha)$-vector is balanced by \cref{lem:balanced} and Hadwiger by \cref{lem:corlofuldisjoint}. Thus, \cref{lem:hadwigerbalanced} holds for any middle $(f,\alpha)$-vectors $x,y$ satisfying $x \preceq y$. Let $x$ be the middle $(f,\alpha)$-vector in the direction $v$ and $\ell_v$ be the corresponding middle $(f,\alpha)$-line. To prove the set $S_1$ contains an open neighborhood $U$ around $v$, we show that for sufficiently small $U$ and any $v' \in U$, the middle $(f,\alpha)$-vector $y$ in the direction $v'$ satisfies $x \preceq y$.

If an entry $x_i$ of $x$ is nonzero, it means that one of $f(\ell_v^+(C_i))$ and $f(\ell_v^-(C_i))$ is \emph{strictly less than} $\alpha$. Fix a closed ball $\mathcal{B}$ containing all sets in $\mathcal{F}$. Since our monotone functions are continuous with respect to Hausdorff distance, given any other line $\ell_{v'}$ with Hausdorff distance sufficiently close to $\ell_{v}$ inside $\mathcal{B}$, we will also have $f(\ell_{v'}^+(C_i)) < \alpha$ if $f(\ell_v^+(C_i)) < \alpha$ (and similarly for $f(\ell_v^-(C_i))$). This gives us that the middle $(f,\alpha)$-vector $y$ corresponding to $\ell_{v'}$ satisfies $x \preceq y$. By \cref{lem:continuous}, the function assigning middle $(f,\alpha)$-lines to directions ($v\mapsto \ell_v \cap \mathcal{B}$) is continuous, so we can indeed find some neighborhood $U$ around $v$ such that for all $v' \in U$, the Hausdorff distance between $\ell_{v'} \cap \mathcal{B}$ and $\ell_v \cap \mathcal{B}$ is sufficiently small.
\end{proof}

Note that a discussion similar to \cref{subsec:4.1} also applies to this colorful setting. 
In particular, we can show that for a family $\mathcal{F}$ satisfying \cref{assump3}, the single arc condition, and the transitivity condition, the existence of a colorful $(f,\alpha)$-consistent ordering of $\mathcal{F}$ is equivalent to the existence of some ordering of $\mathcal{F}$ so that every three sets in $\mathcal{F}$, no two sets of the same color, has an $(f,\alpha)$-transversal that $(f,\alpha)$-stabs the sets in order.
This is because the appropriately modified statements and proofs of Propositions~\ref{prop:single+stab->consord}~and~\ref{prop:single+trans->stab} still go through.

\section{Future work}

We suggest conjectures and problems based on our quantitative Hadwiger-type theorems. As mentioned in the introduction, Pollack and Wenger \cite{Pollack:1990cna} proved the following generalization of Hadwiger's theorem to $\mathbb{R}^d$:

\begin{theorem}[\cite{Pollack:1990cna}]\label{thrm:higherdimHadwiger}
    A finite family $\mathcal{F}$ of connected sets in $\mathbb{R}^d$ has a hyperplane transversal if and only if there is a consistent $k$-ordering of $\mathcal{F}$ for some $0\leq k \leq d-1$.
\end{theorem}

We conjecture a quantitative version of this theorem. For that, we need higher-dimensional analogues of our definitions. 

\begin{definition}
    Let $C_1, \dots, C_n$ be sets in $\mathcal{C}(\mathbb{R}^d)$. Let $f = (f_{1}, \dots , f_{n})$ be a tuple of monotone functions $f_{i}: \mathcal{C}(\mathbb{R}^d) \rightarrow \mathbb{R}$ chosen for each $C_i$. 
    Let $\alpha = (\alpha_{1}, \dots, \alpha_{n}) \in \mathbb{R}^n$. 
    We say a directed hyperplane $H\subseteq \mathbb{R}^d$ is an \textcolor{blue}{$(f,\alpha)$-hyperplane stabber} of the set $C_i$ if $f_{i}(H^+(C_i)), f_{i}(H^-(C_i)) \geq \alpha_{i} $. 
We say $H$ is an \textcolor{blue}{($f,\alpha)$-hyperplane transversal} of sets $C_1, \ldots, C_n$ if it is an $(f, \alpha)$-hyperplane stabber of all of the sets.
As before, by a slight abuse of notation, we let $f$ act as $f_{i}$ and $\alpha$ as $\alpha_{i}$ on the set $C_i$.
With this notation, we say $H$ is an \textcolor{blue}{$(f,\alpha)$-separation hyperplane} if there are some $1\leq i,j\leq n$ such that either  
$f(H^+(C_i)), f(H^-(C_j)) < \alpha$ or $f(H^-(C_i)), f(H^+(C_j)) < \alpha$. 
\end{definition}
Assume that the family $\mathcal{F}\subset \mathcal{C}(\mathbb{R}^d)$ with the monotone function $f$ and value $\alpha$ satisfies the higher-dimensional analogues of \cref{remark:hasudorffmetric} and \cref{rmk:all directions stab}.
Here is a natural generalization of an {$(f,\alpha)$-consistent ordering} to higher dimensions.
\begin{definition}
      We say that a family $\mathcal{F}$ of finite compact convex sets in $\mathbb{R}^d$ has an \textcolor{blue}{$(f,\alpha)$-consistent $k$-ordering} if there exists a map $\phi: \mathcal{F}\to \mathbb{R}^k$ such that for any subfamilies $\mathcal{F}_1, \mathcal{F}_2 \subseteq \mathcal{F}$ with $|\mathcal{F}_1|+|\mathcal{F}_2|\leq k+2$,
  if there is an $(f,\alpha)$-separation hyperplane with sets in $\mathcal{F}_1$ on one side and $\mathcal{F}_2$ on the other, then
  $conv(\phi(\mathcal{F}_1)) \cap conv(\phi(\mathcal{F}_2)) = \emptyset.$ 

\end{definition}
With these definitions, we state our conjecture for a quantitative transversal theorem in $\mathbb{R}^d$:
\begin{conjecture}
    Let $\mathcal{F}$ be a finite family of compact, convex sets in $\mathbb{R}^d$ with an $(f,\alpha)$-consistent $k$-ordering for some $0\leq k\leq d-1$. Then, there exists an $(f,\alpha)$-hyperplane transversal for all sets in $\mathcal{F}$. 
\end{conjecture}

Another direction of future work is to find a necessary and sufficient condition for the existence of $(f,\alpha)$-transversal in $\mathbb{R}^2$, and more generally in $\mathbb{R}^d$. Since there is \cref{ex:three way separated} that shows that the existence of an $(f,\alpha)$-consistent ordering is not a necessary condition for an $(f,\alpha)$-transversal, it would be interesting to find an appropriate weakening of an $(f,\alpha)$-consistent ordering equivalent to the existence of an $(f,\alpha)$-transversal.

\backmatter

\bmhead{Acknowledgments}
This project was done under the mentorship of Pablo Sober\'on as part of the 2021 New York Discrete Math REU, funded by NSF grant DMS 2051026. Carvalho’s research was supported by the KINSC Summer Scholar program by Haverford College. Takahashi's research was supported by the Grinnell College Internship Funding. 
The authors thank the careful revision by three anonymous referees, whose comments greatly improved this manuscript.

\noindent









\bibliography{sn-bibliography}

\end{document}